\documentclass{amsart}
\usepackage{math}
\usepackage{tikz,mathabx,mathrsfs,mathtools,amsrefs,wasysym,comment,lineno,hyperref,fontawesome}
\usetikzlibrary{cd,quotes,angles,decorations,shapes,positioning,
  arrows,automata,calc,decorations.pathreplacing,calligraphy}
\newenvironment{bnftable}{\[\begin{array}{l@{\texttt{ := }}l}}{\end{array}\]}

\newtheorem{mainthm}{Theorem}

\newtheorem{maincor}[mainthm]{Corollary}
\newtheorem{mainobs}[mainthm]{Observation}
\newtheorem{remark}[thm]{Remark}

\newcommand\gfL{{\mathscr L}}

\newcommand\lamp{{\mathcal L}}
\newcommand\sea{{\mathcal S}}
\newcommand\ver{{\mathcal V}}
\newcommand\hor{{\mathcal H}}
\newcommand{\E}{{\mathbb E}}
\newcommand\power{{\mathscr P}}

\newcommand\xqed[1]{%
  \leavevmode\unskip\penalty9999 \hbox{}\nobreak\hfill
  \quad\hbox{#1}}
\newcommand\qee{\xqed{$\fullmoon$}}

\newcommand{\BS}{\mathrm{BS}}
\newcommand{\wee}{X_{\mathrm{tree}}}

\newcommand\shadepath[4]{
  \coordinate (s) at (#3);
  \coordinate (t) at (#4);
  \path[shade,bottom color=#1,top color=#2] ($(s)!0.6pt!90:(t)$) -- ($(t)!0.6pt!-90:(s)$) -- ($(t)!0.6pt!90:(s)$) -- ($(s)!0.6pt!-90:(t)$) -- cycle;
}

\newcommand\faLightbulbON{\tikz{\node at (0,-0.05){\large\textcolor{orange}{\faLightbulbO}};\foreach\a in {-30,0,...,210} {\draw[orange] (\a:0.14) -- (\a:0.25);}}}

\begin{document}
\title{Shifts on the lamplighter group}

\begin{abstract}
    We prove that the lamplighter group admits strongly aperiodic SFTs, has undecidable tiling problem, and the entropies of its SFTs are exactly the upper semicomputable nonnegative real numbers, and some other results. These results follow from two relatively general simulation theorems, which show that for a large class of effective subshifts on the sea-level subgroup, their induction to the lamplighter group is sofic; and the pullback of every effective Cantor system on the integers admits an SFT cover. We exhibit a concrete strongly aperiodic set with $1488$ tetrahedra. We show that metabelian Baumslag-Solitar groups are intersimulable with lamplighter groups, and thus we obtain the same characterization for their entropies.
\end{abstract}

\author{Laurent Bartholdi}
\address{Universit\"at des Saarlandes, Saarbr\"ucken, Germany}
\email{laurent.bartholdi@gmail.com}

\author{Ville Salo}
\address{University of Turku, Turku, Finland}
\email{vosalo@utu.fi}

\date{February 22nd, 2024}
\thanks{L.B. gratefully acknowledges partial support from the ERC AdG grant 101097307}

\maketitle

\section{Introduction}
For a group $G$ and a finite set $A$, a \emph{subshift} is a closed, $G$-invariant subset $X$ of the Cantor set $A^G$. Every subshift is characterized by a collection of \emph{forbidden patterns}: elements of $A^F$ with finite $F\subseteq G$ none of whose translates may appear in elements of $X$. Three classes of subshifts gained prominence: those defined by a finite collection of forbidden patterns (\emph{subshifts of finite type}, SFT); by a recursively enumerable collection (\emph{effective subshifts}; we assume throughout that all groups are countable and have decidable word problem); and \emph{sofic shifts}, see the next paragraph.

Furthermore, much can be learnt by considering \emph{maps} between subshifts. By definition, a subshift $X\subseteq A^G$ is \emph{sofic} if it is a quotient of an SFT $Y\subseteq B^G$ via a map $B\to A$. If $\phi\colon G\twoheadrightarrow Q$ is a quotient map then every subshift $Y\subseteq B^Q$ may be \emph{pulled back} to a subshift $\phi^*(Y)$ on $G$, whose elements are constant on all fibres of $\phi$; and if $\iota\colon H\hookrightarrow G$ is a subgroup inclusion then every subshift $X\subseteq A^H$ may be \emph{induced} to a subshift $\iota_*(X)$ on $G$ consisting of independent copies of $X$ on each coset of $H$. This is just the $G$-shift with same forbidden patterns as $X$. We consider more generally \emph{$G$-systems}, which are closed subsets of Cantor space endowed with a $G$-action.

Clearly the most desirable subshifts are the SFTs, then the sofic ones, and finally the effective ones. The situation we consider, restricted to subshifts, is in the following picture:
\[\begin{tikzcd}
    1\arrow{r} & H\arrow{r}{\iota} & G\arrow{r}{\phi} & Q\arrow{r} & 1\\
    & \text{$H$-subshifts}\arrow{r}{\iota_*} & \text{$G$-subshifts} & \text{$Q$-subshifts}\arrow{l}[above]{\phi^*}\\
    & \text{$H$-effectives}\arrow[hook]{u}\arrow[dashed]{r}[above]{?} & \text{$G$-sofics}\arrow[hook]{u} & \text{$Q$-effectives}\arrow[dashed]{l}[above]{?}\arrow[hook]{u}
  \end{tikzcd}
\]
and we study whether there are dotted arrows that complete the diagram (in which case they are uniquely determined as restrictions of $\iota_*$ and $\phi^*$).

Several results in the literature may be interpreted as the construction of a rightmost dotted arrow; e.g.~\cites{hochman:recursive,aubrun-sablik:simulation,durand-romashchenko-shen:1d2d} in the case $1\to\Z\to\Z^d\to\Z^{d-1}\to1$ for $d\ge1$, and~\cite{barbieri-sablik:semidirect} for $Q$-effective systems in split extensions $1\to\Z^2\to G\to Q\to1$; see also~\cite{barbieri:simulation}. The reason behind such results is that the extra $\Z$ component(s) may be used as ``time'', in which to run Turing machine computations. For this reason, such results are often called ``simulation theorems''.

If $Q$ is non-amenable and the extension is a direct product, \cite{barbieri-sablik-salo:soficity} constructs a leftmost dotted arrow. If $Q$ is amenable, though, we should not expect there to be such a map:
\begin{mainobs}[= Theorem~\ref{obs:effectivenonsofic2}]\label{obs:effectivenonsofic}
  Let $\iota\colon H\hookrightarrow G$ be a subgroup inclusion, and assume that $G$ is amenable and $H$ is infinite and has decidable word problem. Then there exists an effective $H$-subshift $X$ such that $\iota_*(X)$ is not sofic.
\end{mainobs}
We recover that for $1\to\Z\to\Z^2\to\Z\to1$ there are $\Z$-effective subshifts that do not induce to $\Z^2$-sofic subshifts; it is actually conjectured in~\cite{guillon-jeandel:icc} that only sofic $\Z$-shifts may induce to a sofic $\Z^2$-sofic subshift. We shall soon see that sometimes a large subclass of $H$-effectives may be induced to $G$.

In the extreme case that $H=1$ (or similarly $Q=1$), our question asks whether all $G$-effective subshifts are sofic, in which case $G$ is called \emph{self-simulable}. This happens for instance if $G=F_2\times F_2$, and in numerous other cases, see~\cite{barbieri-sablik-salo:selfsimulable}.

Let us assume that $G$ is amenable, so self-simulation is ruled out, and positive answers to our question entail an infinite ``time'' component $H$. The critical case seems naturally to be a ``small'' infinite, finitely generated $Q$, and an \emph{infinite, locally finite} $H$, namely an infinite subgroup all of whose finitely generated subgroups are finite. This is the situation that we address in this article, in a specific but possibly generic example, the lamplighter group $\lamp$. It is the extension
\[\begin{tikzcd}
    1\arrow{r} & {\displaystyle\bigoplus_{\Z+\frac12}\Z/2}\arrow{r}{\iota} & \lamp\arrow{r}{\phi} & \Z\arrow{r} & 1.
  \end{tikzcd}
\]
(The group $\Z$ acts on the $(\Z+\frac12)$-indexed sum by shifting; see~\S\ref{ss:lamplighter} for a few words on this strange-looking convention.) Our first main result concerns effective \emph{$G$-systems}, analogous to $G$-shifts except that the alphabet is a Cantor set $A^\omega$ rather than a finite set (see~\S\ref{ss:subshifts} for details):

\begin{mainthm}[= Theorem~\ref{thm:effectiveZ2}]\label{thm:effectiveZ}
  Let $X$ be an effective $\Z$-system. Then $\phi^*(X)$ admits an SFT cover.
\end{mainthm}

(In particular, if $X$ is a subshift, $\phi^*(X)$ is a sofic $\lamp$-subshift.) The kernel of $\phi$ is the locally finite group $\sea=(\Z/2)^{(\Z+\frac12)}$, which further decomposes as $\sea=\hor\times\ver$ for $\hor = (\Z/2)^{(-\N-\frac12)}$ and $\ver = (\Z/2)^{(\N+\frac12)}$. We denote by $\eta\colon \sea \to \hor$ the natural projection. We can also simulate Turing machines within $\ver$, proving the following result which is the main ingredient in the proof of Theorem~\ref{thm:effectiveZ}, and in view of Observation~\ref{obs:effectivenonsofic} is probably the most that can be hoped for: denoting by $\iota^-\colon\hor\hookrightarrow\lamp$ the inclusion, there exists an effective $\hor$-subshift $X$ such that $(\iota^-)_*(X)$ is not sofic, and an effective $\sea$-subshift such that $\iota_*(X)$ is not sofic.

\begin{mainthm}[$\subseteq$ Theorem~\ref{thm:effectiveH2}]\label{thm:effectiveH}
  Let $X$ be an effective $\hor$-system. Then $\iota_*\eta^*(X)$ admits an SFT cover.
\end{mainthm}

(In particular, if $X$ is a subshift, $\iota_*\eta^*(X)$ is a sofic $\lamp$-subshift.) We deduce a number of corollaries for the lamplighter group. The \emph{domino problem} for a group $G$ asks for an algorithm that, upon input a finite collection of forbidden words, determines whether the corresponding SFT is non-empty.
\begin{maincor}[= Corollary~\ref{cor:undecidable2}]\label{cor:undecidable}
  The domino problem on $\lamp$ is undecidable.
\end{maincor}
(The best result known beforehand was that the \emph{seeded domino problem}, asking for an algorithm deciding, given a finite collection of forbidden words and a cylindrical set, whether the corresponding SFT intersects it non-trivially~\cite{bartholdi-salo:ll}).

\begin{maincor}[= Corollaries~\ref{cor:examples2}, \ref{cor:examples3}, \ref{cor:examples4} and~\ref{cor:examples5}] \label{cor:examples}
  There exists on $\lamp$
  \begin{enumerate}
  \item  a strongly aperiodic SFT; namely a non-empty subshift on which $\lamp$ acts freely;
  \item an SFT none of whose configurations is recursively enumerable;
  \item a non-empty SFT each of whose configurations has Kolmogorov complexity $2^{\Theta(n)}$ on each of its radius-$n$ balls;
  \item an SFT with entropy $\eta$ for every $\Pi^0_1$ nonnegative real number $\eta$.
  \end{enumerate}
\end{maincor}

The best result known beforehand was that there exists a non-empty subshift without periodic points~\cite{cohen:lamplighters}, namely in which each $\lamp$-orbit is infinite. We also provide an explicit construction, based on Kari's transducer method~\cite{kari:smallwang}, of a strongly aperiodic SFT with $1488$ legal size-$4$ patterns, and based on the same method a proof of the undecidability of the tiling problem.

\subsection{Strategy}
We develop the theory of substitutional subshifts on locally finite groups such as $\sea$. It is somewhat analogous to substitutional subshifts on $\Z$ or $\Z^2$, though we do not have a general theory containing both. Keeping in mind Mozes's result~\cite{mozes:tilings} that every substitutional $\Z^2$-subshift is sofic, we prove related statements about $\sea$-subshifts.  We then use substitutional subshifts to define a ``scaffolding'' on which to build simulations. As a side effect, we prove the following variant of Theorem~\ref{thm:effectiveH}:
\begin{mainthm}[= Theorem~\ref{thm:substitutive2}]\label{thm:substitutive}
  Let $X$ be a substitutive $\sea$-shift. Then $\iota_*(X)$ is sofic.
\end{mainthm}

The starting point in the proof of Theorem~\ref{thm:effectiveH} is that when $X\subseteq A^\hor$ is a substitutive subshift, $(\iota^-)_*(X)$ is sofic. This allows simulating tiling systems independently on all $\sea$-cosets.

The next step in the proof is an adaptation of the construction by Durand-Romashchenko-Shen of self-similar tilesets~\cite{durand-romashchenko-shen:fpa} from the theory of plane tilings to the locally finite case. Let us briefly recall the idea. Given a tileset $\sigma$ on a group $G$, we construct another tileset $\tau(\sigma)$, whose tiles are bigger than the tiles of $\sigma$, and which ``implements'' $\sigma$, in the sense that $\tau(\sigma)$ encodes a program checking whether a tuple of colours represents a tile of $\sigma$. If the construction $\sigma\mapsto\tau(\sigma)$ is algorithmic, then it admits a fixed point $\sigma=\tau(\sigma)$, essentially a ``Quine'' (a program that prints itself) in the realm of subshifts. More precisely, this is a ``self-similar'' program, checking that it runs itself on arbitrarily large tiles, and performing on the side any desired calculation. In this manner arbitrary Turing machines may be simulated, by having them run on tiles of all sizes. We devote Section~\ref{sec:FP} to an outline the fixed point construction for our tiling systems, and Section~\ref{sec:Implementation} to an implementation of this scheme using a ``universal tile set'', including some engineering details that we could not find elsewhere.

The geometry of the grid $[0,2^n-1]\times[0,2^n-1]$ in $\Z^2$ and that of $(\Z/2)^n\times(\Z/2)^n$ are very different; nevertheless, we show that the natural bijection between these two sets, given by binary encoding, allows Turing machine calculations to be represented in both cases.

Rather than relying on general fixed-point properties, we actually construct the self-similar tileset. This has the benefit of giving us control on the entropy of the construction, which is quite low (in particular, sublinear).

\section{Definitions and notations}
For a group $H$ and a set $S$ we write $H^{(S)}$ for the group of functions $f\colon S \to H$ whose support $\{s \in S \mid f(s) \neq 1\}$ is finite, under the pointwise product $(f \cdot f')(s) = f(s) \cdot f'(s)$.

If furthermore $G$ is a group acting on $S$, the \emph{wreath product} $H\wr_S G$ is the extension $H^{(S)}\rtimes G$, with action $(f\cdot g)(s)=f(g s)$ of $G$ on $H^{(S)}$. When $S=G$ with action by left translation, we simply write $H\wr G$.

All the groups we consider are assumed to be \emph{effectively countable}: there is a bijection between the group and $\omega$ such that, via this bijection, the multiplication table $\omega\times\omega\to\omega$ is computable.

\subsection{The lamplighter group}\label{ss:lamplighter}
The only wreath product we will consider in this article is the ``lamplighter group'' $\lamp=(\Z/2)\wr\Z$. It is convenient to set $\mathbb E=\Z+\frac12$, identified with the unit intervals with integer endpoints, and to write $\lamp = (\Z/2)^{(\mathbb E)}\rtimes\Z$.

The name ``lamplighter group'' comes from the following interpretation. There is a lamp at every integer-plus-a-half position on the line, and each lamp may be ON or OFF. The lamplighter stands at an integer position. An element of the lamplighter group is a task for the lamplighter: she may move up and down the line an integer number of steps, and in doing so she may toggle a lamp she passes next to.

Seen from the perspective of the lamplighter, there is a natural action of $\lamp$ on the Cantor space of lamp configurations $(\Z/2)^{\mathbb E}$, which we call the \emph{lamp action}; her motions amounts to shifting lamp configurations, possibly toggling the lamp that passes between positions $\tfrac12$ and $\tfrac{-1}2$.

Evidently every such task is a composition of elementary tasks (and their inverses), which are of two kinds: $a$ ``move up one step'', and $b$ ``move up one step, flipping the lamp encountered along the way''. (This explains our at first sight strange convention for the lamp positions: if they were also at integral positions, there would be ambiguity as to whether the departure or arrival lamp is to be flipped). Thus $\{a,b\}$ is a generating set for $\lamp$, and in fact $\lamp$ admits a presentation
\[\lamp=\langle a,b\mid (a^nb^{-n})^2\;\forall n\ge1\rangle.\]
Using these generators, the Cayley graph of $\lamp$ (namely, the graph with vertex set $\lamp$ and edges from $g$ to $ag$ and to $bg$ for every $g\in\lamp$) looks as in Figure~\ref{fig:sealevel}, with edges $a$ in green and $b$ in red.
\begin{figure}[!ht]
  \centerline{\begin{tikzpicture}[x={(5cm,0cm)},y={(-2cm,-1cm)},z={(0cm,6cm)},>=stealth']
    \draw[thick,->] (-1.1,0,-0.2) -- node[pos=0.33,left] {$\Z$} +(0,0,1.2); 
    \draw[thick,->] (-1.15,0,0.4) -- node[below] {$\hor$} +(0.4,0,0); 
    \draw[thick,->] (-1.1,0.1,0.4) -- node[above] {$\ver$} +(0,-0.8,0); 
    \pgfmathsetmacro\h{4}
    \pgfmathsetlengthmacro\rad{3mm*pow(2,-\h)}
    \pgfmathsetmacro\hmo{\h-1}
    \foreach\k in {0,...,\h} {
      \pgfmathsetmacro\z{\k/\h}
      \pgfmathsetmacro\cz{1-\z}
      \pgfmathsetmacro\maxl{pow(2,\k)-1}
      \pgfmathsetmacro\maxm{pow(2,\h-\k)-1}
      \ifnum\k=2
        \foreach\x in {-0.5,-0.166,0.166,0.5} {
          \draw[blue,dashed] (\x,-0.6,\z) -- (\x,0.6,\z);
          \draw[blue,dashed] (-0.6,\x,\z) -- (0.6,\x,\z);
        }
      \fi
      \foreach\m in {0,...,\maxm} {
        \pgfmathsetmacro\y{-1+\m*2/(pow(2,\h-\k)-0.9999)}
        \foreach\l in {0,...,\maxl} {
          \pgfmathsetmacro\x{1-\l*2/(pow(2,\k)-0.9999)}
          \ifnum\k=\h\else
          \foreach\color/\pos in {green/1,red/0} {
            \pgfmathsetmacro\nx{1-(2*\l+(\pos ? mod(\m,2) : 1-mod(\m,2)))*2/(pow(2,\k+1)-1)}
            \pgfmathsetmacro\ny{-1+floor(\m/2)*2/(pow(2,\h-(\k+1))-0.9999)}
            \pgfmathsetmacro\nz{(\k+1)/\h}
            \pgfmathsetmacro\ncz{1-\nz}
            \pgfmathsetmacro\shade{30+20*\y}\def\shadecolor{\color!\shade}
            \pgfmathsetmacro\nshade{30+20*\ny}\def\nshadecolor{\color!\nshade}
            \shadepath{\shadecolor}{\nshadecolor}{\z*\x,\cz*\y,\z}{\nz*\nx,\ncz*\ny,\nz}
          }
          \fi
          \filldraw[fill=white] (\z*\x,\cz*\y,\z) circle (1.5pt);
        }
      }
    }
  \end{tikzpicture}}
  \caption{A tetrahedron in $\gfL$, with in blue the ``sea level'' grid $\sea=\hor\times\ver$; it is represented with $\hor$ going across the page and $\ver$ going into it.}\label{fig:sealevel}
\end{figure}

We single out important subgroups of $\lamp$: first, there is a natural map $\phi\colon\lamp\to\Z$, extracting the movement of the lamplighter. Its kernel is $\sea=(\Z/2)^{(\E)}$, which naturally decomposes as $\sea=\hor\times\ver$ with $\hor=(\Z/2)^{(-\N-\frac12)}$ and $\ver=(\Z/2)^{(\N+\frac12)}$. 

Even though the group $\sea$ is locally finite and therefore very far from the grid $\Z^2$, its representation in the Cayley graph of Figure~\ref{fig:sealevel} appears like a grid, represented by the overlaid dashed blue edges (which may connect arbitrarily distant edges in the Cayley graph). We refer to $\sea$ as the ``sea level'' of the Cayley graph, with $\hor$ and $\ver$ respectively the ``horizontal'' and ``vertical'' directions.

In fact, when coding configurations on the sea-level, we prefer to use the four directions North, South, West, East, based on the standard map orientation (North pointing up), viewing the sea-level ``from above''. The Up and Down directions refer to movement of the lamplighter, assuming the street is oriented up-down.

We finally name some group homomorphisms. We denote by $\iota\colon\sea\hookrightarrow\lamp$ the natural inclusion, and write similarly $\iota^-\colon\hor\hookrightarrow\lamp$ and $\iota^+\colon\ver\hookrightarrow\lamp$. There is a natural projection $\eta\colon\sea\twoheadrightarrow\hor$.

The Cayley graph of the lamplighter group belongs to a family of graphs called \emph{Diestel-Leader graphs}. The \emph{full binary tree} is the two-way infinite tree $\mathcal T_2$, namely the $3$-regular tree oriented with constant out-degree $1$. It comes naturally with a height function $\beta\colon\mathcal T_2\to\Z$, compatible with the edge orientation. Then the Cayley graph in Figure~\ref{fig:sealevel} may be described by its vertex set $\{(x,y)\in\mathcal T_2^2\mid\beta(x)+\beta(y)=0\}$, with an edge from $(x,y)$ to $(x',y')$ precisely when there is an edge from $x$ to $x'$ and an edge from $y'$ to $y$.

More generally, the Diestel-Leader graph $\mathsf{DL}(p,q)$ has vertex set $\{(x,y)\in\mathcal T_p\times\mathcal T_q\mid\beta(x)+\beta(y)=0\}$ and edges as above. It may also be described directly as follows as an edge-labeled graph: its vertex set consists of pairs $(x,n)$ with $n\in\N$ and $x\colon(\Z+\tfrac12)\to\Z$ of finite support, taking values in $\{0,\dots,p-1\}$ at positions $<n$ and values in $\{0,\dots,q-1\}$ at positions $>n$. There is an edge labeled $(x,y)$ from $((\dots,x,\dots),n)$ to $((\dots,y,\dots),n+1)$ with the `$x$' and `$y$' at position $n+\tfrac12$. If $p=q$, then this graph is the Cayley graph of $\Z/p\wr\Z$; edges with label $(x,y)$ correspond to the generator ``shift and add $y-x$ to the lamp just crossed'' of $\Z/p\wr\Z$. In all cases, this labeling $\mathsf{DL}(p,q)$ is at least as rich as the Cayley graph labeling, so all our claims easily extend to graph subshifts on $\mathsf{DL}(p,q)$.

\subsection{Subshifts, induction and pullback}\label{ss:subshifts}
Let $G$ be a group. A \emph{subshift} on $G$ is a closed, $G$-invariant subset of $A^G$, for a finite set $A$; the action of $G$ on $A^G$ is given by $(g x)(h)=x(h g)$. For a finite set $F\subseteq G$ and $p\in A^F$ called a \emph{forbidden pattern} we define
\[S_p=\{x\in A^G\mid gx\restriction_F \neq p\quad\forall g\in G\},\]
namely the subshift of those configurations on $G$ that contain no translate of the pattern $p$. A \emph{subshift of finite type} (SFT) is a finite intersection of $S_p$'s. A \emph{sofic shift} is a factor of a subshift of finite type (on a possibly different alphabet $B$) by a coding map $B\to A$. An \emph{effective subshift} is an intersection of $S_p$'s for a recursively enumerable sequence of forbidden patterns $p$.

More generally, an \emph{(effective) Cantor system} is an (effective) subshift, but with the finite set $A$ replaced by a Cantor set $A^\omega$, and with patterns only checking finitely many coordinates: a pattern is $p\in (A^n)^F$ for some $n\in\N$ and $F\subseteq G$ finite, with corresponding
\[S_p=\{x\in(A^\omega)^G\mid\forall g\in G:\exists f\in F:(g x)(f)\notin p(f)A^\omega\}.\]
A Cantor $G$-system $K$ is isomorphic to a subshift precisely when it is \emph{expansive}, that is, there is a neighbourhood of the diagonal in $K$ such that whenever $x\neq y\in K$ the orbit $\{(g x,g y)\mid g\in G\}$ escapes that neighbourhood.

A subshift of \emph{almost finite type} (AFT) is the factor of an SFT by a covering map that is almost everywhere $1$-to-$1$: there exists an SFT $Y$ with covering map $\pi\colon Y\to X$ such that the set of points in $X$ with unique preimage has full measure with respect to every invariant Borel probability measure on $X$. This definition does not require minimality of $X$, but in that case implies that the set of points in $X$ with unique preimage is comeagre (dense $G_\delta$). See~\cite{hochman:recursive}*{page~134} for a discussion.

If $K$ is a finite abelian group, and $X, Y$ are $G$-dynamical systems for some group $G$, we say that $X$ is a \emph{$K$-extension} of $Y$ if there is a bijection $f\colon X \to Y \times K$ and a cocycle $\eta\colon G \times Y \to K$ such that if $f(x) = (y,k)$ then $gx = (k + \eta(g, x), gy)$. Observe in particular that $X$ is then an extension of $Y$ with all fibers of cardinality $\#K$. Now generalizing the idea of the previous paragraph, we say $Y$ is \emph{almost $K$-to-$1$} if there is a $K$-extension of $Y$ which is AFT. (This is indeed a direct generalization: if $1$ denotes the trivial group, an $1$-to-$1$ extension is just a topological conjugacy, so almost $1$-to-$1$ still means AFT in the new sense.)

Consider two groups $F,G$ and a homomorphism $\phi\colon F\to G$. Then, on the one hand, $G$-systems induce $F$-systems by pull-back: if $G\curvearrowright Y$, set $\phi^*(Y) = Y$ as a set, with $F$-action $f\cdot y=\phi(f)\cdot y$.  On the other hand, $F$-systems induce $G$-systems by induction: if $F\curvearrowright X$, set $\phi_*(X)=\{y\colon G\to X\mid y(\phi(f) g)=f y(g)\;\forall f\in F,g\in G\}$, with natural action $(g\cdot y)(g')=y(g' g)$.

Let us consider these operations for subshifts, namely $Y\subseteq B^G$ and $X\subseteq A^F$. In case $\phi$ is surjective, the pull-back $\phi^*(Y)$ is again a subshift,
\begin{equation}\label{eq:pullback}
  \phi^*(Y)=\{y\circ\phi\in B^F\mid y\in Y\}.
\end{equation}
In case $\phi$ is injective, the induction of $X$ is again a subshift,
\begin{equation}\label{eq:induction}
  \phi_*(X)=\{y\in A^G\mid \forall g\in G: y(\phi(-) \cdot g)\in X\};
\end{equation}
the verification is straightforward, the embedding $\phi_*(X)\hookrightarrow A^G$ being given by $(y\colon G\to X)\mapsto(g\mapsto y(g)(1))$.

We briefly note that the pull-back $\phi^*(Y)$ of a $G$-system by an injective map $\phi$ is called a \emph{subaction}; the pull-back of a subshift has no reason to be expansive, e.g.\ the pull-back of any infinite subshift to the trivial subgroup. The literature sometimes mentions ``projective subactions'' of subshifts; again for $\phi$ injective and $Y\subseteq B^G$, this is the $F$-shift given by~\eqref{eq:pullback}. It is the factor of the subaction given by restriction to a single copy of $\phi(F)$ in $B^G$.

We found out that there was a wide diversity in the simulation literature, where results of the following form are stated: \emph{``consider a subgroup $H\le G$ satisfying `\underline{\qquad}'. Then every effective $H$-subshift $X$ is a factor of a projective $H$-subaction $X$ of a sofic $G$-shift''}. In all the cases we witnessed, it turned out that $H$ was also a quotient of $G$, and the proofs actually gave the stronger statement that the pull-back of $X$ is sofic. See~\cites{aubrun-sablik:simulation,barbieri-sablik:semidirect,crumiere-sablik-schraudner:speed,hochman:recursive,barbieri:simulation} for the articles we have found in the literature that actually prove more than they claim. We exclusively use pull-backs and induction in this text, hoping this clarifies the setting somewhat.

\section{Subshifts on locally finite groups}\label{ss:locfin}
Consider a finite group $F$, and its restricted product $\hor=F^{(\N)}$, the group of finitely-supported functions $\N\to F$. There is a natural map $\sigma\colon \hor \to \hor$, the one-sided shift: $\sigma(f_0,f_1,\dots)=(f_1,\dots)$. For every $N \subset \N$, we can naturally consider $F^{(N)}$ as a subgroup of $F^{(\N)}$ (for finite $N$ we may drop the parentheses). For $n \in \N$ we write $F^n$ for $F^{\{0,1,...n-1\}}$. For $f\in F,n\in\N$ we denote by `$f_{@n}$' the element of $\hor$ with a single `$f$' in position $n$.

\begin{defn}[Substitutional shifts]
  A \emph{substitutional subshift} on $\hor$ is a set $X_\tau\subseteq A^\hor$
  given by a map $\tau\colon A\to \power(A^F)$ as follows. The map $\tau$ induces a map $\tilde\tau\colon A^\hor \to \power(A^\hor)$ by
  \[\begin{split}\tilde\tau(x)&=\{x'\colon (f_0\mapsto x'(f_0,f_1,\dots))\in\tau(x(f_1,\dots))\quad\forall(f_1,\dots)\in\hor\}\\
      &=\{x'\mid (x'(f_{@0}h))_{f\in F}\in\tau(x(\sigma h))\quad\forall h\in F^{(\N_{>0})}\}.\end{split}\]
  In other words, for a configuration $x\in A^\hor$ the new value at $h \in \hor$ is computed (non-deterministically) by $\tau$ from the old value at $\sigma(h)$ and the first co\"ordinate of $h$ via the rule $\tau$, so if $y\in\tilde\tau(x)$ then we have $y(-,f_1,\dots)\in\tau(x(f_1,\dots))$. Set then
  \[X_0=A^\hor\text{ and }X_{n+1}=\bigcup_{f\in F}f_{@0}\tilde\tau(X_n)\text{ for all }n\in\N,\]
  and finally $X_\tau=\bigcap_{n\in\N}X_n$.
\end{defn}

\begin{lem}
  The $X_n$ are $\hor$-invariant and nested.
\end{lem}
\begin{proof}
  By induction: $X_0$ is clearly $\hor$-invariant, and the $X_n$ are nested. If $X_n$ is $\hor$-invariant, then $\tilde\tau(X_n)$ is $F^{(\N_{>0})}$-invariant, and $\hor = F \times F^{(\N_{>0})}$ so $X_{n+1}$ is $\hor$-invariant.
\end{proof}

With every $x\in X_\tau$ is associated one or more \emph{branching directions} $b\in F^\N$: if $x\in b_{@0}\tilde\tau(X_\tau)$, say $x=b_{@0}\tilde\tau(x')$ for some $b\in F,x'\in X_\tau$, and $(b',b'',\dots)$ is a branching direction of $x'$, then $(b,b',b'',\dots)$ is a branching direction of $x$.

Let $\hor_n=\ker(\sigma^n) \cong F^n$ be the subgroup of $\hor$ consisting of the first $n$ factors; it is generated by all $f_{@0},\dots,f_{@(n-1)}$ with $f\in F$. Naturally $\hor=\bigcup_{n\in\N}\hor_n$. It is easy to see that, for $m\ge n$, the values of $\tilde\tau^n(x)$ on $\hor_m$ depend only on those of $x$ on $\hor_{m-n}$. Thus $X_\tau$ is also the topological closure of the $\hor$-orbit of fixed points of $\tilde\tau$.

We think of $\tau$ as a non-deterministic function, and when $\#\tau(a) = 1$ for all $a$, we write $\tau(a) = \{b\}$ simply as $\tau(a) = b$. This should not cause confusion.

\begin{exple}[``Sunny-side-up'']\label{ex:ssu}
  The ``sunny-side-up'' $\hor$-shift consists of those sequences $x\in\{*,\dagger\}^\hor$ with $x(h)=*$ for at most one position $h\in\hor$. It is deterministic substitutional: for $A=\{*,\dagger\}$ identify $A^F$ with $A\times\dots\times A$ and set $\tau(\dagger)=(\dagger,\dots,\dagger)$ and $\tau(*)=(*,\dagger,\dots,\dagger)$.

  Note that if $x\in X_\tau$ has a single $*$ value, then its branching direction precisely determines where it is. If $x=\dagger^\hor$ then every $b\in F^\N$ is a branching direction.
\end{exple}

\begin{exple}[``Thue-Morse'']\label{ex:tm}
  Set $A=\{0,1\}$, and recall that the classical Thue-Morse shift is the substitutional subshift $Y\subset A^\Z$ associated with the substitution $0\mapsto01,1\mapsto10$. It is the closure of the word $0110100110010110\dots$ and is aperiodic.

  In analogy, consider $F=\Z/2$ so $\hor=(\Z/2)^{(\N)}$ and define $\tau\colon A\to A^{\Z/2}=A^2$ by $\tau(0)=(0,1)$ and $\tau(1)=(1,0)$. The corresponding subshift $X_\tau$ contains exactly two elements. More precisely, consider the homomorphism $\sigma\colon\hor\to\Z/2$ given by $\sigma(h)=\sum_{n\in\N}h(n)$; then the action of $\hor$ on $X_\tau=\Z/2$ factors via $\sigma$.

  Indeed, there is a natural (set-wise!) bijection between $\hor$ and $\N$, identifying $(f_0,f_1,\dots)\in \hor$ with the binary expansion $f_0+2f_1+\cdots$, when each $f_i$ is treated as an element of $\{0,1\}$. Then given $x=(x_n)\in A^\N$, under this bijection we have $\tilde\tau(x)=y$ with $(y_{2n},y_{2n+1})=\tau(x_n)$ for all $n\in\N$ or $(y_{2n+1},y_{2n})=\tau(x_n)$ for all $n\in\N$.

  Writing $A=\Z/2$, we have $\tau(a)=(a,1-a)$, so the actions of $A$ and $F$ are intertwined; this extra symmetry makes $X_\tau$ periodic, but does not appear in the classical $\Z$-system which is only ``almost periodic''.
  
  Still under this bijection, the elements of $X_\tau$ are $01101001\dots$ and $10010110\dots$; they may be written as $\tilde\tau^\infty(0^\hor)$ and $\tilde\tau^\infty(1^\hor)$ respectively. \qee
\end{exple}

There is a straightforward visualization of substitutional shifts via \emph{coset trees}. The collection $\bigsqcup_{n\ge0}\hor_n\backslash \hor$ of all cosets of $\hor_n$ in $\hor$ naturally forms a ``corooted binary tree'' $\mathscr T_F$, with leaves the cosets of $\hor_0=1$ and an edge between $\hor_n h$ and $\hor_{n+1}h$ for all $n\in\N,h\in \hor$. Since $\hor_n\backslash \hor\cong \hor$ via $\sigma^n$, we may also write the vertex set of $\mathscr T_F$ as $\hor\times\N$, with an edge between $(h,n)$ and $(\sigma(h),n+1)$ for all $h\in H,n\in\N$.

Elements of $X_\tau$ may then be given by \emph{parse trees}. Consider pairs $(b,x)$ with $b\in F^\N$ an infinite sequence of elements of $F$ and $x\colon\mathscr T_F\to A$ a labeling of the corooted $\#F$-regular tree. There is a natural action of $\hor$ on such pairs, in which $\hor$ acts trivially on $F^\N$ and acts on $\mathscr T_F$ by right multiplication on cosets; this last action reads $(g x)(h,n)=x(h\sigma^n(g),n)$ in the notation above. Again writing $\hor\times\N$ the vertices of $\mathscr T_F$, set
\[X_\tau'=\{(b,x)\mid (x((f b_n)_{@0}h,n))_{f\in F}\in\tau(x(\sigma h,n+1))\text{ for all }h\in \hor,n\in\N\}.\]
In effect, we are storing in $X_\tau'$ an entire prehistory of applications of the substitution $\tau$, as well as making explicit as $b$ a choice of branching direction of $(b,x)$.

\begin{lem}
  The set $X_\tau'$ is closed and $\hor$-invariant, and the natural map sending $(b,x)$ to the restriction of $x$ to $\mathscr T_F$'s leaves is a factor map $X_\tau'\twoheadrightarrow X_\tau$.
\end{lem}
\begin{proof}
  It is clear that $X_\tau'$ is closed. The $\hor$-invariance follows since if $(b,x) \in X_\tau'$ and $g\in\hor$ then
  \[\begin{split}((g x)((f b_n)_{@0}h,n))_{f\in F}&=(x((f b_n)_{@0}h\sigma^n(g),n))_{f\in F}\in\tau(x(\sigma(h\sigma^n(g)),n+1))\\&=\tau(x(\sigma(h)\sigma^{n+1}(g)),n+1))=\tau((gx)(\sigma h),n+1).
    \end{split}\]
  Consider next $x\in X_\tau$; so there exist $b_0,b_1,\dots\in F$ and $x=x_0,x_1,x_2,\dots\in X_\tau$ with $(b_n)_{@0}x_n\in\tilde\tau(x_{n+1})$ for all $n\in\N$. Extend the labeling of $\mathscr T_F$'s leaves by $x$ into a labeling, still written $x$, of all its vertices by $x(h,n)=x_n(h)$, and note that $(b,x)$ projects to $x$, and defines an element of $X_\tau'$ since
  \[\begin{split}(x((f b_n)_{@0}h,n))_{f\in F}&=(x_n((f b_n)_{@0}h))_{f\in F}\\&\in\tau((x_{n+1})(\sigma h))=\tau(x(\sigma h,n+1)).\end{split}\]

  Conversely, given $(b,x)\in X_\tau'$, note that the restrictions $h \mapsto x(h,n)$ of $x$ define elements $x_n$ of $A^\hor$ satisfying $(b_n)_{@0}x_n\in\tilde\tau(x_{n+1})$ for all $n\in\N$, so in particular $x_n\in X_\tau$ for all $n\in\N$ and thus $x_0\in X_\tau$.
\end{proof}

Subshifts of finite type, or sofic shifts, are entirely uninteresting on $\hor$ since $\hor$ is locally finite. However, quotients of substitutional subshifts under letter encodings $A\to B$, and intersections with full shifts $A_0^\hor$ for some $A_0\subset A$, are of great interest, as we shall soon see.

\subsection{Almost odometric subshifts}
There is a factor map $X'_\tau\twoheadrightarrow F^\N$ mapping $(b,x)$ to the branching direction $b$; this map is equivariant with respect to the action of $\hor$ on $X'_\tau$ and the natural action of $\hor$ on its closure $F^\N$, which we call the \emph{odometer}.

\begin{defn}[Uniquely decodable subshifts]\label{def:uniquelydecodable}
  The substitution $\tau$ is \emph{uniquely decodable} if the factor map above descends to a factor map $X_\tau\twoheadrightarrow F^\N$. This holds if every $x\in X_\tau$ has a unique branching direction.

  We call $X_\tau$ \emph{almost odometric} if furthermore there is an almost everywhere $1$-to-$1$ factor map to $F^\N$. Recall, as before~\cite{hochman:recursive}*{page~134}, that ``almost everywhere'' refers to a full-measure event with respect to every invariant Borel probability measure on $X_\tau$. This definition does not require minimality of $X$, but in that case implies that the set of points in $X$ with unique preimage is comeagre (dense $G_\delta$).
\end{defn}

For example, the Thue-Morse subshift from Example~\ref{ex:tm} is not odometric, because it is not even faithful. An example that gets closer, with $F=\Z/2$, is $\tau(t) = (t, t+1)$ on the alphabet $\Z/3$. In this example, the natural odometer is a factor: from $x\in X_\tau$ we recover $b_n$ from $x(1_{@n})-x(0)-1$. There is a natural action of $\Z/3$ on $X_\tau$ by addition of a constant, and which acts freely on fibres. Presumably no factor map is almost $1$-to-$1$, though we have not shown this. The following example will be important for us:
\begin{exple}[``Period doubling'']\label{ex:pd}
  Consider $F=\Z/2$ and $A=\Z/2$, and the deterministic substitution $\tau\colon A\to A^2$ given by $\tau(a)=(0,1-a)$.
\end{exple}

It is the $\hor$-analogue of the ``period-doubling'', or ``ruler'', subshift, generated by the infinite word $0100010101000100\dots$ with $n$th digit the parity of the largest power of $2$ dividing $n$. As a $\Z$-shift, it is a quotient of the Thue-Morse subshift by $\Z/2$, but as an $\hor$-shift it is quite different:

\begin{lem}\label{lem:pd}
  The period-doubling subshift $X_\tau$ is almost odometric.
\end{lem}
\begin{proof}
  Given a configuration in $X_\tau$, we can work out its branching direction $b$ locally while going up the coset tree and finding where the value $0$ is. Sometimes, we must look two levels up: when we see $(0,0)$ on some level, we cannot deduce a bit of $b$ --- but then in its neighbour we will see $(0,1)$ or $(1,0)$ and from there deduce the bit of the branching direction.

  This substitution has a \emph{coincidence}: the first symbol of $\tau(t)$ is always the same. It follows that if we know the branching direction $b$ and $b_n = 0$, then we know the substitutive images in the first $\hor_n$-block. Now consider any invariant measure on $X_\tau$. It pushes forward to an invariant measure on the odometer, which is well-known to be unique, namely the uniform Bernoulli measure. We obtain that the branching direction is picked according to the uniform Bernoulli measure, thus the probability that $b_n = 0$ for infinitely many $n$ is $1$.
\end{proof}

\section{Subshifts on the lamplighter group}
\label{sec:SubstitutionConstructions}
We briefly recall that elements of the lamplighter group $\lamp$ may be represented as pairs $(s,n)$ with $s\colon\E\to\Z/2$ finitely supported and $n\in\Z$. We also recall the maps $\iota\colon\sea\hookrightarrow\lamp$ and $\iota^-\colon\hor\hookrightarrow\lamp$ and $\iota^+\colon\ver\hookrightarrow\lamp$ and $\eta\colon \sea \hookrightarrow \hor$ and $\phi\colon \lamp \twoheadrightarrow\Z$.

In this section we prove several variants of the claim ``substitutive $\sea$-subshifts induce to sofic shifts on $\lamp$'': 
\begin{prop}\label{prop:soficH}
  Let $X\subseteq A^\hor$ be a substitutional shift. Then $(\iota^-)_*(X)$ is sofic.
\end{prop}

\begin{prop}\label{prop:soficHinduction}
  Let $X\subseteq A^\hor$ be a substitutional shift. Then $\iota_*\eta^*(X)$ is sofic.
\end{prop}

\noindent We also consider a slight variant of substitutional subshifts:
\begin{defn}[$Y_\tau$]\label{defn:ytau}
  Let $\tau\colon A\to\power(A^2)$ be a substitution. Define $Y_{\tau} \subset A^{\sea}=A^{\ver\times\hor}$ as follows:
  \[Y_\tau=\{y\colon\hor\times\ver\to A\mid \exists b\in\{0,1\}^\omega:\;\forall v\in\ver:\;y(v,-)\in X_\tau\text{ with branching direction }b\}\]
  Thus $Y_{\tau}\restriction_{\hor\times\{0\}} = X_{\tau}$, and for every $y \in Y_{\tau}$, there exists $b \in \{0,1\}^\omega$ such that $v \cdot y$ admits the same branching direction $b$ for all $v \in \ver$.
\end{defn}

\begin{prop}\label{prop:equalbranchingstructures}
  Let $\tau\colon A \to\power(A^2)$ be a substitution. Then $\iota_*(Y_{\tau})$ is sofic.
\end{prop}

Of these, Proposition~\ref{prop:equalbranchingstructures} will be used as black box in our first proof of Theorem~\ref{thm:effectiveH} (combined with the analogous results for the other embedding $\iota^+$ of $\ver$ in $\lamp$, at positions $-\frac12-\N$; direct products; intersections with subalphabets, and quotients). All these variants have very close proofs.

We finally prove the following result, which while not directly used for anything, is the starting point for our construction of a ``small'' aperiodic tileset in~\S\ref{sec:Small}:
\begin{prop}\label{prop:soficZ}
  Let $X \subseteq A^\Z$ be a sofic shift. Then $\phi^*(X)$ is sofic.
\end{prop}

\subsection{Marking corooted binary trees}
\label{sec:Corooted}

We recall that tilesets on $\lamp$ are described by their neighbourhood: ``spiders'' specify legal edge colours on the neighbouring edges in directions $(a,b,a^{-1},b^{-1})$ of each vertex; and ``tetrahedra'' specify legal vertex colours at vertices $\{v, a v, b^{-1}a v, b v\}$ along each relation $b^{-1}a b^{-1}a v=v$. Note that the set of valid tetrahedra that eventually appear in the SFT is closed under the flip $(x,y,z,w) \leftrightarrow (z,w,x,y)$ which describes the same tetrahedron but viewed from vertex $b^{-1}a v$ rather than from $v$; if we only specify one orientation of a tetrahedron, the interpretation is that the other one is also allowed. We describe a spider tile as a list $(c_a,c_b,c_{a^{-1}},c_{b^{-1}})$ of colours, in that order.

Consider first the tileset $\Theta_0$ with edge colours $C_0=\Z/3$ and allowed spiders $(c,c,c+1,c)$ and $(c,c,c,c+1)$ for all $c\in C_0$:
\[\Theta_0=\{(c,c,c+1,c), (c,c,c,c+1)\mid c\in C_0\}.\]
Thus at each vertex there is a single edge in the $\phi$-decreasing direction that is distinguished, namely the one with colour one more (modulo 3) than its three neighbours. We can also view the configurations as vertex colourings, by colouring every vertex $v$ of the Cayley graph of $\lamp$ by the $c$ appearing three times on the edges touching $v$. This mapping from edge colourings to vertex colourings is obviously injective, since the colour of an edge is determined by its endpoint in the decreasing $\phi$-direction.

Recall from~\S\ref{ss:locfin} that the corooted binary tree has as vertex set $\N\times\N$, with edges connecting $(m,n+1)$ to $(2m,n)$ and $(2m+1,n)$ for all $m,n\in\N$. Its leaves are $\N\times\{0\}$.
\begin{lem}\label{lem:theta_0}
  The SFT on $\lamp$ defined by the tileset $\Theta_0$ has the following structure: for each $c\in C_0$, the subgraph spanned by the $c$-coloured edges is a disjoint collection of corooted binary trees, with possibly a single full binary tree.
\end{lem}

One configuration in this SFT comes from the vertex colouring $x\colon\lamp\to C_0$ defined by
\[x(s,n)=\#\{k>n:s_k=1\}\bmod 3\quad\text{ for all }s\in(\Z/2)^{(\mathbb E)},n\in\Z.\]
As a typical model of a corooted binary tree appearing in this colouring, consider the span of the vertices $\{(s,n)\mid n\le0,s(k)=0 \; \forall k>n\text{ except }s(\frac12)=1\}$, which is coloured $1$. Note that its leaves are naturally identified with the subgroup $\hor$, via the bijection between $\hor$ and $\N$ mentioned in Example~\ref{ex:tm}. There is, additionally, one full binary tree, $\{(s,n)\mid s(k)=0 \; \forall k>n\}$ which is coloured $0$.

\begin{proof}
  Consider a configuration $x$ in the SFT. The tileset $\Theta_0$ imposes that every vertex has two distinct colours $c,c+1$ on the edges below it, and the two edges above it are coloured $c$. Thus in the subgraph spanned by $c$-coloured edges, every vertex has degree $1$ or $3$, and the connected components are trees growing upwards. What remains to be seen is that if such a tree has a leaf, then its set of leaves is precisely the set of its vertices at the same height.

  This follows from the following claim: ``\emph{in every height-$k$ tetrahedron, there is a height-$k$ monochrome upwards-growing binary tree, say of colour $c$; at each level $\ell\in\{0,\dots,k\}$ there is a height-$(k-\ell)$ binary tree forking off it, coloured $c+1$; from each level of these trees, there is a binary tree forking off it, coloured $c+2$; etc.}'' (Our convention is that the height of a tetrahedron is the number of edges on a path from bottom to top.) Indeed, if a monochrome tree in $x$ has a leaf, then it is a secondary tree in an increasing, nested sequence of tetrahedra containing that leaf, and the claim enforces a whole level of leaves.

  This claim itself is proven by induction; the case $k=0$ being trivial. Consider next a height-$k$ tetrahedron $T$, say between $\phi$-positions $N$ and $N+k$. By induction, the two sub-tetrahedra of $T$ between positions $N$ and $N+k-1$ satisfy the claim, so they contain height-$(k-1)$ binary trees respectively coloured $c$ and $d$. Now the upwards-pointing edges at position $N+k-1$ are also respectively coloured $c$ and $d$; and the rule $\Theta_0$ forces $c\neq d$, so (up to swapping $c$ and $d$) $d=c+1$. Thus $T$ contains a height-$k$ binary tree coloured $c$, and the claim holds for $T$.
\end{proof}

These corooted binary trees will serve as parse trees for the induced shifts, in the proofs of the propositions. We may mark the non-leaf vertices of each tree with a finite amount of data $B$ so that all vertices at same height share this data: set $C_{1,B}=C_0\times B$, and allow spider colourings $((c,b),(c,b),(c+1,b'),(c,b))$ and $((c,b),(c,b),(c,b),(c+1,b'))$ at each spider for all $c\in C_0,b,b'\in B$, calling the resulting tileset $\Theta_{1,B}$. Consider now a vertex $v$ which is a non-leaf in a tree $\mathcal T$. The edge at $v$ which does not belong to $\mathcal T$ is a leaf of a new corooted binary tree $\mathcal T'$, and it carries the data $b'$ which uniformly colours $\mathcal T'$; thus all vertices at same height as $v$ in $\mathcal T$, which are all the leaves of $\mathcal T'$, share the data $b'$.

\begin{proof}[Proof of Proposition~\ref{prop:soficH}]
Let $\tau\colon A\to\power(A^2)$ be a substitution. Consider the set of colours $C=C_{1,\{0,1\}}\times A$, tagging each corooted binary tree with a bit, and allow colours $((c,b,a_0),(c,b,a_1),(c+1,b',a'),(c,b,a))$ and $((c,b,a_0),(c,b,a_1),,(c,b,a),(c+1,b',a'))$ for all $c\in C_0,b,b'\in\{0,1\},a,a'\in A$ whenever $(a_{b'},a_{1-b'})\in\tau(a)$. Then the $A$-labels on every corooted binary tree form a valid parse tree for the substitutional shift, so their leaves form a valid element of $X_\tau$. These collections of leaves are precisely the cosets of $\hor$ in $\lamp$, and are coloured independently. At each vertex, there is a unique edge with different $c$-colour than the other three edges; label the vertex with the $A$-value on that edge. This presents $(\iota^-)_*(X_\tau)$ as a sofic shift.

Note that the colours specify an extra tree, which is not corooted; its labels are simply ignored by this last map.
\end{proof}

\begin{proof}[Proof of Proposition~\ref{prop:soficHinduction}]
Let $X \subseteq A^\hor$ be a substitutional shift. Proposition~\ref{prop:soficH} shows that $Y\coloneqq (\iota^-)_*(X)$ is sofic. Now consider the substitutive $\ver$-subshift defined by the deterministic substitution $\tau\colon A \to A^2, \tau(a) = (a,a)$, so $\#X_\tau = \#A$ and $X_\tau$ consists of fixed-points for the $\ver$-action. As before $(\iota^+)_*(X_\tau)$ is sofic, where $\iota^+\colon \ver \to \lamp$ is the natural embedding, since $\iota^+$ and $\iota^-$ are isomorphic. Now the intersection $Y \cap (\iota^+)_*(X_\tau)$ is also sofic, and coincides with the desired subshift $\iota_*(\eta^*(X))$.
\end{proof}

\begin{proof}[Proof of Proposition~\ref{prop:equalbranchingstructures}]
Let $X\subseteq A^\hor$ be a substitutional shift. Proposition~\ref{prop:soficH} shows that $Y\coloneqq (\iota^-)_*(X)$ is sofic. Now, similarly to what we did in the proof of Proposition~\ref{prop:soficHinduction}, synchronize the $\Z/3$-elements and bits in each $\ver$-coset, by overlaying the induction of a finite subshift of $\ver$-fixed points, this time over the alphabet $\Z/3 \times \{0,1\}$.
\end{proof}

\begin{proof}[Proof of Proposition~\ref{prop:soficZ}]
We begin by proving that the pullback of the full shift $A^\Z$ is sofic. We start by the tileset $\Theta_{1,A}$ on $C_{1,A}=\Z/3\times A$ as above, in which every corooted tree shares an element of $A$. We set then $C_{2,A}=C_{1,A}\times C_{1,A}$, with valid spider colourings those that are valid on the first $C_{1,A}$; that are valid on the second $C_{1,A}$, after inverting generators (thus in the second co\"ordinate we accept all $((c+1,a'),(c,a),(c,a),(c,a))$ and all $((c,a),(c+1,a'),(c,a),(c,a))$); and that have the same $A$-value on the distinguished edge of the first and second factors.

In this manner, each vertex $v$ is a leaf of a corooted binary tree in the positive $\phi$-direction, and also a leaf of a corooted binary tree in the negative $\phi$-direction. Each of these binary trees is tagged with the same $A$-label, which is therefore shared by the whole fibre $\phi^{-1}(\phi(v))$, and these fibres are independently coloured by $A$. The map assigning to each vertex $v$ the unique $A$-colour on the distinguished edges emanating from $v$ expresses the pullback of the full shift $A^\Z$ as a sofic shift on $\lamp$.

Consider now an arbitrary sofic shift $X\subseteq A^\Z$. Without loss of generality, $X$ is the quotient of an SFT $Y\subseteq B^\Z$ defined by two-letter legal words $F\subseteq B^2$ via a map $B\to A$ (so $Y$ is an \emph{edge shift}). Explicitly, $Y=\{y\in B^\Z\mid (y_n,y_{n+1})\in F\;\forall n\}$. Decorate then every vertex of $\lamp$ with an element of $B$, and impose, on each spider, that these extra decorations be of the form $(b',b',b,b)$ for some $(b,b')\in F$. Overlaying this construction with the subshift on $C_{2,B}$ above yields an SFT presentation for the subshift $\phi^*(X)$.
\end{proof}

In the SFTs of this section, the trees are quite wild, and the full dynamics are difficult to describe. In Section~\ref{ss:W}, we tame the trees, by running $\Theta_0$ simultaneously up- and downward, and synchronizing the two processes. This is only strictly needed for our strongest simulation result, but may be of independent interest to some readers.

\subsection{A simple $\sea$-aperiodic tile set}\label{sec:SmallSea}

It seems that tile sets with no nontrivial $\sea$-periods are quite common, and in our search for strongly aperiodic tile sets, we found many of them. In this section, we exhibit the simplest one we know.

Let $\Delta_{\text{count}}$ be the following set
 
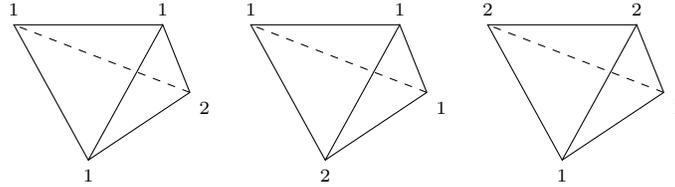
\[ \begin{tikzpicture}[scale=0.9]
    \foreach\label/\o/\a/\ab/\b/\shift in {tree/{1}/{1}/{2}/{1}/-0.5cm,instruction/{2}/{1}/{1}/{1}/3cm,value/1/2/1/2/6.5cm} {
      \begin{scope}[xshift=\shift]
        \coordinate[label={below:$\scriptstyle\o$}] (A) at (0,0);
        \coordinate[label={below right:$\scriptstyle\ab$}] (B) at (1.5,1);
        \coordinate[label={above:$\scriptstyle\a$}] (C) at (-1.1,2);
        \coordinate[label={above:$\scriptstyle\b$}] (D) at (1.1,2);
        \draw (A) -- (B) -- (D) -- (A) -- (C) -- (D);
        \draw[dashed] (B) -- (C);
      \end{scope}
    }
  \end{tikzpicture} \]
of three tetrahedra (already closed under the permutation $(1 \; 3)(2 \; 3)$).

\begin{lem}
\label{lem:SimpleAperiodic}
The SFT defined by allowed tetrahedra $\Delta_{\text{count}}$ is nonempty, and has no configurations with a nontrivial period in $\sea$.
\end{lem}

\begin{proof}
For nonemptiness, we prove that there exists a valid tetrahedron with only $a$s on the top row, for $a \in \{1, 2\}$. This is clear trivial for $k = 0$. For larger $k$, if $a = 1$, construct a height $k-1$ tetrahedron with $1$s on top, and one with $2$'s on top, and position these under the row; for $a = 2$, use two tetrahedra with $1$s on the top row.

A brief look at the set of legal tetrahedra shows that in a tetrahedron of height $k$, the matrix $M = \begin{pmatrix} 1 & 2 \\ 1 & 0 \end{pmatrix}$ tells how the number of symbols evolves when we move downward in a tetrahedron. Eigenvalue decomposition shows that, writing $(a, b)$ for respectively the number of $1$-symbols and $2$-symbols on the bottom column of a height $k$ tetrahedron, we have $(a, b) = 2^k (2, 1)/3 - (-1)^k (-1, 1)/3$ if any symbol on the top row is $1$, and $2^k (2, 1)/3 + (-1)^k (-1, 1)*2/3$ if any symbol is $2$. These vectors do not agree for any $k$, so the top row is always constant.

This also already shows that the SFT has no $\ver$-period for any configuration: if there were a configuration with period $v \in \ver$, then taking the minimal $n$ with $v \in \ver_k$ we see that if the height $k$ tetrahedron above the identity has $1$'s on its top row, then $v$ exchanges the bottom columns of a tetrahedron with $2$'s on top, and one with $1$'s on top, and then shuffles the columns; this cannot preserve the configuration due to the different letter counts. If the tetrahedron of height $k$ has $2$'s on the top row, then the tetrahedron above it, of height $k+1$, must have $1$'s on the top row, thus, the tetrahedron of height $k$ ``behind'' the first one we considered gives a similar contradiction.

In fact, the previous argument shows that for all configurations $x$ in this SFT, and any $v \in \ver \setminus \{1\}$ we have $vx|_\sea \neq x|_\sea$. Due to the fact right $\hor$-cosets are constant in all configurations, we have $hx|_\sea = x|_\sea$ for all $h \in H$. Thus, no point has period $s \in \sea \setminus \{1\}$ with nontrivial projection to $\ver$. On the other hand, the set of possible periods of points in a dynamical system is closed under conjugation, and the normal closure of every $h \in \hor \setminus \{1\}$ intersects $\ver$.
\end{proof}

To check that this is indeed the smallest example, we note first there is no nontrivial SFT with smaller alphabet, which would have size $1$. The number of tetrahedra is also optimal, in that one can quickly work through all pairs (and singletons) of tetrahedra closed under the permutation $(1 \; 3)(2 \; 4)$, and conclude that the only non-trivial example implements the Thue-Morse substitution in the sense of Example~\ref{ex:tm}, and the resulting SFT has exactly two configurations.

\subsection{Aperiodic tilesets and undecidability}\label{sec:Small}

We now use the tilesets constructed above to solve two problems: the construction of a strongly aperiodic tileset, namely one all of whose tilings have trivial isotropy group, and the proof of undecidability of the tiling problem. These are direct proofs of, respectively, Corollary~\ref{cor:examples}(1) and Corollary~\ref{cor:undecidable} from the Introduction. (This section is independent from the proofs of our main simulation results, whose proofs begin in the following section).

The method, for both problems, is to represent in $\lamp$ an infinite binary tree by $\Z$-translates of $\hor$, and to use $\ver$ and trees in the opposite direction to connect vertices at the same height. In this manner, we realize in $\lamp$ a union of ``discrete hyperbolic planes'': the $1$-skeleton of the tiling of $\mathbb H^2$ by pentagons with vertex set $2^\Z(\Z+i)$.

We then apply the technique of Kari~\cite{kari:revisited}, storing an element of $\R$ or $\R^2$ on each translate of $\hor$, represented as an average (Beatty sequence); and forcing by the tileset each such element to be an affine image of the previous one. The existence of aperiodic piecewise-affine maps on $\R$, respectively the undecidability of the mortality problem for piecewise-affine maps on $\R^2$, let one conclude the proofs.

We have elected to represent the tilesets by tetrahedra; we are therefore considering vertex colourings, by a set $C$ of colours; and a tileset is a subset $\Delta$ of $C^{\{1,a,a b^{-1},b\}}$, with entries written in that order. It determines the subshift
\[X_\Delta=\{x\colon\lamp\to C\mid (x(v),x(v a),x(v a b^{-1}),x(v b))\in\Delta\text{ for all }v\in\lamp\}.\]
Note that if $(c_1,c_2,c_3,c_4)\in\Delta$ then we may assume without loss of generality $(c_3,c_4,c_1,c_2)\in\Delta$ (by applying the condition to $v a b^{-1}$); we only specify half of the allowed tiles, assuming implicitly that they are completed by the above switch operation.

\begin{thm}\label{thm:aperiodic}
  There exists a strongly aperiodic tileset consisting of $1488$ tetrahedra.
\end{thm}

\begin{thm}\label{thm:undecidable}
  The domino problem on $\lamp$ is undecidable.
\end{thm}

The proofs of these two results begin in the same manner. As colours $C$ we consider a product $C=\Z/2 \times I^3\times A\times B^2$, for some finite sets $I,A,B$. The slots are respectively called the \emph{count}, \emph{instruction}, \emph{value} and \emph{carry}.

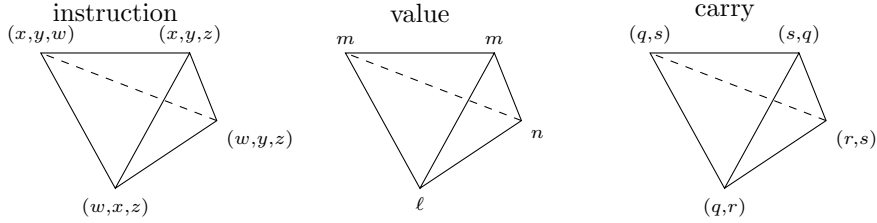
\begin{figure}[h]
  \begin{tikzpicture}[scale=0.9]
    \foreach\label/\o/\a/\ab/\b/\shift in {instruction/{(w,x,z)}/{(x,y,w)}/{(w,y,z)}/{(x,y,z)}/1.5cm,value/\ell/m/n/m/6cm,carry/{(q,r)}/{(q,s)}/{(r,s)}/{(s,q)}/10.5cm} {
      \begin{scope}[xshift=\shift]
        \coordinate[label={below:$\scriptstyle\o$}] (A) at (0,0);
        \coordinate[label={below right:$\scriptstyle\ab$}] (B) at (1.5,1);
        \coordinate[label={above:$\scriptstyle\a$}] (C) at (-1.1,2);
        \coordinate[label={above:$\scriptstyle\b$}] (D) at (1.1,2);
        \draw (A) -- (B) -- (D) -- (A) -- (C) -- (D);
        \draw[dashed] (B) -- (C);
        \node at (0,2.6) {\label};
      \end{scope}
    }
  \end{tikzpicture}
  \caption{The tetrahedra labellings (first component omitted)}
\end{figure}

In the count slot, we allow the SFT from the previous section:
\begin{equation}\label{eq:count}
  \Delta_{\text{count}} = \{ (1, 1, 2, 1), (2, 1, 1, 1), (1, 2, 1, 2) \}.
\end{equation}
The role of this component is simply to ensure that no configuration is $\sea$-periodic (Lemma~\ref{lem:SimpleAperiodic}) in our ``small'' aperiodic tile set. It is likely that the other components also automatically ensure $\sea$-aperiodicity.

Consider next the instruction slot. In it, we allow
\[\Delta_{\text{instruction}} = \{((w,x,z),(x,y,w),(w,y,z),(x,y,z))\mid w,x,y,z\in I\}.\]
This has the effect of specifying an element in $I$ across every level $\phi^{-1}(n)$. Indeed view the numbers $w,x,y,z$ on the four faces of the tetrahedron. The value $y$ propagates in the second co\"ordinate along an upwards-growing tree, the value $z$ propagates in the third co\"ordinate along a downwards-growing tree, and the values $w,x$ are copies along removed edges respectively from an upwards- or downwards-growing tree. It follows that the first co\"ordinate is constant across every level. For a colour $c$ we call that first co\"ordinate $i(c)\in I$. 

The value slot is simply an element of some set $A$, written $v(c)$, and we allow
\[\Delta_{\text{value}} = \{(\ell,m,n,m)\mid \ell,m,n\in A\}.\]
In words, we force the two values at the top of a tetrahedron to be identical. In the carry slot, we allow
\[\Delta_{\text{carry}} = \{((q,r),(q,s),(r,s),(s,q))\mid q,r,s\in B\}.\]
This has the effect of specifying an element between any two neighbouring vertices on a column of every level. Indeed on each downwards-growing binary tree two values are propagated, a ``front'' and a ``back'' one. Across every vertex, the $(q,s)$ values are turned into a $(q,r)$ and an $(r,s)$ value. (The ``$(s,q)$'' in the last position is there because the meaning of ``front'' and ``back'' are reversed at that node). For a colour $c$ we respectively call $p^<(c)$ and $p^>(c)$ the first and second co\"ordinates of its carry slot.

For now we have not specified any constraint between the tree, instruction, value and carry slots. This we do as follows. Consider a vector space $V=\R$ or $\R^2$, and a family of affine transformations $f_i(x)=M_i x+b_i$ on $V$, one for each $i\in I$. Let $A,B$ be finite subsets of $V$. We allow tetrahedron colourings
\[\Delta=\Big\{(c_1,c_a,c_{a b^{-1}},c_b)\mid f_{i(c_1)}(v(c_a))+p^<(c_1)=\frac{v(c_1)+v(c_{a b^{-1}})}2 + p^>(c_{a b^{-1}})\Big\}.\]
In words, we force each tetrahedron to map the value on top, via $f_{i(c_1)}$, to the average of the values on the bottom, give or take some carries.

These tiling constraints mean the following: on every column of every level, there are some values in $V$. Assume that these values have a well-defined average; then every such average is the $f_i$-image of the average above it.

\begin{proof}[Proof of Theorem~\ref{thm:aperiodic}]
  We specialize further to $V=\R$ and $I=\{0,1\}$, with $f_0(x)=2x$ and $f_1(x)=2x/3$. We take $A=\{0,1,2\}$ and $B=\{-1,-2/3,-1/3,0,1/3,2/3\}$, and restrict the colour set $C$ to consist of those $c\in(\Z/3)^2\times I^3\times A\times B^2$ such that if $i(c)=0$ then $v(c)\ne0$ and $p^<(c),p^>(c)\in\{-1,0\}$ while if $i(c)=1$ then $p^<(c),p^>(c)\in\{-1/3,0,1/3,2/3\}$.

  The classical arguments of Kari~\cite{kari:smallwang} apply: there exist bi-infinite orbits of $\{f_0,f_1\}$ on $[1/2,2]$, but no periodic orbit. Thus, in any tiling of $\lamp$ by this tileset, there does not exist any period with non-trivial $\phi$-image. On the other hand, already the tiling by $\Delta_{\text{tree}}$ admits no period with trivial $\phi$-image. It follows that no tiling has any period.

  We finally crudely estimate the number of tiles. The constraints $\Delta_{\text{count}}$ allow $3$ tetrahedra, taking into account the symmetry; the constraints $\Delta_{\text{instruction}}$ allow $2^4$ tetrahedra; the constraints $\Delta_{\text{value}}$ allow $3^3$ tetrahedra; and the constraints $\Delta_{\text{carry}}$ allow $6^3$ tetrahedra so there are $6^7$ tetrahedra to consider. A simple computer program (ours was written in \texttt{Julia}) tests which of those satisfy the compatibility conditions between $i(c),p^<(c),p^>(c),v(c)$ and the linear condition, and answers $1488$.
\end{proof}

\begin{proof}[Proof of Theorem~\ref{thm:undecidable}]
This is analogous to the proof in \cite{kari:revisited}. We derive from the above construction a reduction of the mortality problem for piecewise affine maps: \textit{``Does a given system of rational affine transformations $f_1,\dots,f_n$ of the plane and disjoint unit squares $U_1,\dots,U_n$ with integer corners have an immortal starting point?''} which Kari proves undecidable in \cite{kari:revisited} (based on the mortality problem of Turing machines, whose undecidability is due to Hooper \cite{hooper:tmimmortality}).
\end{proof}

\section{From effective subshifts to sofic shifts}\label{sec:Real}

We now proceed to the proof of our main results. We introduce a variant of the classical Wang tilings adapted to our locally finite groups. Recall that $\sea$ denotes the elementary abelian group $(\Z/2)^\E$, and write $\hor_n=(\Z/2)^{\{-1/2,\dots,1/2-n\}}$ and $\ver_n=(\Z/2)^{\{1/2,\dots,n-1/2\}}$ and $\sea_n=\hor_n\times\ver_n$. For any bit sequence $b\in(\Z/2)^\E$ we have bijections $\eta_b\colon\hor_n\to\{0,\dots,2^n-1\}$ and $\nu_b\colon\ver_n\to\{0,\dots,2^n-1\}$ given by
\begin{equation}\label{eq:etanu}
  g\in\hor_n\mapsto\sum_{i=-1/2}^{1/2-n}(g_i\oplus b_i)2^{-1/2-i},\qquad
  g\in\ver_n\mapsto\sum_{i=1/2}^{n-1/2}(g_i\oplus b_i)2^{-1/2+i}.
\end{equation}
Taken together, $\eta_b\times\nu_b$ gives a bijection $\sea_n\to\{0,\dots,2^n-1\}^2$.
\begin{defn}[Wang towers]
  Let $D$ be a finite set. A \emph{Wang tileset} is a subset $B\subseteq D^{N,E,S,W}\cong D^4$, and a \emph{Wang tower} is a colouring $x\colon\sea\to B$ such that, for some bit sequence $b\in(\Z/2)^\E$, one has
  \[\forall (s',s''),(t',t'')\in\hor\times\ver:\;\begin{cases}
    x(s',s'')_E=x(t',t'')_W & \text{ if }\eta_b(s')+1=\eta_b(t')\text{ and }s''=t'',\\
    x(s',s'')_N=x(t',t'')_S & \text{ if }s'=t'\text{ and }\nu_b(s'')+1=\nu_b(t'').
  \end{cases}\]
  In words, $x$ defines colourings of $\sea_n$ for all $n\in\N$, such that the colouring of $\sea_{n+1}$ consists of colourings of the four cosets of $\sea_n$ in $\sea_{n+1}$, assembled as the four panes of a window in an order specified by $b$.
\end{defn}
Equivalently via the bijection $\eta_b\times\nu_b$, a Wang tower is a tower of tilings of squares $\{0,\dots,2^n-1\}^2$ in such a manner that each square is contained in one of the four quadrants of the next one. Yet equivalently, it is a tiling of a plane, half-plane or quadrant. Thus a Wang tileset admits a tower if and only if it tiles the plane; though for a given Wang tileset the \emph{space} of Wang towers is only vaguely related to the space of Wang tilings of the plane, so these notions should be treated as distinct.

Our argument proceeds in two steps: firstly, we show that every Wang tileset may be encoded into a sofic shift on $\lamp$, in such a manner that the configurations visible on cosets of $\sea$ are precisely the Wang towers. We then show how arbitrary effective shifts may be encoded by Wang towers.

\subsection{From Wang tilesets to sofic shifts}\label{ss:wang2sofic}
Let $B\subseteq D^4$ be a Wang tileset. We prove here that \emph{there is a subshift of finite type on $\lamp$ that, on every coset of $\sea$, projects to a Wang tower for $B$.}

Choose an almost odometric substitution $\tau_0\colon A_0\to\power(A_0^2)$, consider the alphabet $A=A_0\times B^2$ and define the substitution $\tau\colon A\to\power(A^2)$ by
\[\tau\colon(a,(b,b')) \mapsto \{((\tau_0(a)_1, (b,b'')), (\tau_0(a)_2, (b'',b'))) \mid b'' \in B\}.\]

By Proposition~\ref{prop:equalbranchingstructures}, there is a sofic shift $Z_1 \subseteq A^\lamp$ with $Z\restriction_\sea = Y_\tau$. Recall from Definition~\ref{defn:ytau} that $Y_\tau$ is the $\sea$-subshift in which each $\hor$-coset contains an element of $X_\tau$, and the branching directions of all these elements are synchronized. Switching the roles of $\hor$ and $\ver$, there is a sofic shift $Z_2$ whose restriction to $\sea$ is the $\sea$-subshift in which each $\ver$-coset contains an element of $X_\tau$, again with synchronized branching directions.

Consider a configuration in $Z_1$, and specifically the contents on one of the subgroups $\hor_n$. After appropriate translation in $\lamp$, we may assume that the first component of the configuration in $\hor_n$ is $\tau_0^n(a)\in A_0^{\hor_n}$ for some $a \in A$. We refer to passing to this unique shift as \emph{normalization}. Since $\tau$ is uniquely decodable, only the standard branching direction (corresponding to the lexicographic order) can produce this content, so the second component (in $(B^2)^{\hor_n}$) also uses the standard branching direction. We perform the exact same normalization for a configuration in $Z_2$.

Projecting to the second component, every node in $\sea$ contains symbols $(b,b')\in B^2$ coming from $Z_1$ and $(b''',b'')\in B^2$ coming from $Z_2$. We add the constraint $b = b'''$ to our SFT, and declare $b$ to be the symbol at the node, and $b'$ and $b''$ to be respectively the symbols at the east and north neighbours. One can easily check that $\tau$ precisely transfers this information. We thus impose the additional constraints $b_E=b'_W$ and $b_N=b''_S$.

\begin{rem}
  We note that this already proves Corollary~\ref{cor:undecidable}: indeed, it is undecidable, given a Wang tileset, whether it tiles arbitrarily large squares. Together with Theorem~\ref{thm:undecidable} this already provides our second proof this fact.
\end{rem}

\begin{rem}
  At the east and north ends of $\sea_n$ (after normalization) the neighbour information may be meaningless: this happens if and only if these nodes are actually the north and/or east border of $\sea_m$ for all $m \geq n$ (i.e.\ normalization keeps moving them to the north/east boundary). This is a key difference between classical Wang tilings and Wang towers. We note intuitively (and later rigorously) that this happens with ``probability zero'', so for many purposes this can be ignored.
\end{rem}

\begin{rem}
  We can also make the Wang tileset aware of the $A_0$-colouring, formally by having subsets $B_a \subseteq B$ for each $a \in A_0$, and adding the requirement that the Wang tile paired with $a \in A_0$ belongs to $B_a$. In this manner, the Wang tiles may access a finite amount of information on their co\"ordinate: identifying $\sea_n$ with $2^n\times2^n$, its tiling has access to a configuration over $A_0$ whose value at $(i,j)$ is $(\tau_0^n(a)_i, \tau_0^n(a')_j)$ for some $a,a'\in A_0$.
\end{rem}

\subsection{From effective subshifts to Wang tilesets}

Let $X\subseteq C^\hor$ be any effective subshift. We shall construct a Wang tileset whose horizontal configurations represent $X$.

Let us say a few words about what general subshifts on $\hor$ look like. By definition, they are topologically closed subsets of $C^\hor$ that are closed under the action of the locally finite group $\hor$. They are defined by forbidden patterns, which can be taken to have domains $\hor_n$. Unlike on $\Z$, where the different shifts of a pattern overlap by different amounts, with our $\hor_n$-domain convention two forbidden patterns cannot overlap nontrivially without having full overlap. On the other hand,, a single forbidden pattern with domain $\hor_n$ has $2^n$ different orientations in which we check for it, in each $\hor_n$-coset. 

We begin with an informal discussion of the construction. Our Wang tileset will be a product, where on the bottom layer we have a symbol $aa' \in A^2$, in the middle layer we have a symbol $c \in C$, and on the top layer we have a large number of tiles used to implement everything we need. The bottom layer is simply matched against the configuration in the underlying substitutional subshift from the previous section. The $C$-component is what we will project to in the end, so our logic should check that the contents of this comes from $\eta^*(X)$ with $X$ the effective subshift we started with. Thus, we should check that the $C$-values remain equal when moving in the south and north directions, which we can easily do with Wang tile rules directly by adding another layer (on this layer, send the colour $c$ south and north, and simply send $0$ west and east).

We will construct a Wang tileset in which each configuration on a large enough square splits into ``macrotiles of level $k$'', which further split into macrotiles of level $k-1$, and so on (the tiles themselves are thought of as level-$0$ macrotiles). On the macrotiles of level $k$, we overlay a power of the substitution $\tau \times \tau$ (aligned with the position of the macrotile). The macrotiles should also have access to at least some values on the $C$-component, and should be capable of performing universal computation on these values. We then simply enumerate the forbidden patterns of $X$ so that every forbidden pattern $p \in C^{\hor_m}$ gets enumerated in all large enough squares $\sea_n$, and in these squares we check that we do not see the pattern $p$ nor any of its $\hor_n$-shifts in the part of the pattern on the $C$-component that we have access to. It is also important that every $\hor_n$-coset is considered by arbitrarily large macrotiles.

Of course the Wang tiles do not have direct access to the group, so by checking the $\sea_n$-shifts of $p$, what we really mean is that a square on $\sea_n$ will shift $p$ with respect to its own numbering, interpreting the horizontal positions on the square in binary and taking the natural bijection with $\hor_n$.

We claim that if we perform such a construction, we are done. To see this, first note that indeed a forbidden pattern $p$ of $X$ cannot appear anywhere: in a normalized configuration on $\sea_n$ large enough that a large square is constructed there, we explicitly check that $p$ does not appear if it fits the area the square considers, and by assumption every area is considered by some macrotile. In a non-normalized configuration we of course check the same constraint, since if the normalizing shift is $h^{-1}$ then we are simply checking against translates all $gh \cdot p, g \in \hor_n$ instead of all translates $g \cdot p, \hor_n$, and this is the same set. Conversely, if $x \in X$, then no matter what the substitutive structure is, we can simply overlay the macrotiles according to the substitutive structure, and write $x_h$ as the $C$-colour at every element of $\eta^{-1}(h)$ and none of the forbidden patterns will match with the contents of $x$.

There are two known ways to implement this general plan, the Aubrun-Sablik construction (based on Robinson tiles) and the fixed-point tilesets of Durand-Romashenko-Shen. We will use the latter, describing it at a high level in the next subsection and in more detail in~\S\ref{sec:Implementation}.

\subsection{The fixed point construction}\label{sec:FP}
In this section, we outline the fixed point construction of~\cites{durand-romashchenko-shen:fpa,durand-romashchenko-shen:1d2d} in our setting. 
Consider two Wang tilesets $B,B'$ and an integer $N\ge2$. A \emph{simulation} of $B$ by $B'$ is an injective map
\[ S\colon B\to\{\text{tilings of $\{1,\dots,N\}^2$ by }B'\} \]
preserving side matchings and such that every $B'$-tiling is uniquely decomposable into $N\times N$-squares, called \emph{macrotiles}. Consider a $B'$-tiling $x$ all of whose $N\times N$-squares are in the image of $S$, and let $S^{-1}(x)$ denote the $B$-tiling obtained by replacing each such square by its $S$-preimage. Every tile in $x$ is called a \emph{child} of the tile in $S^{-1}(x)$ it corresponds to, while that tile is called its \emph{parent}.

A simulation of $B$ may be produced as follows. Imagine that $B$ is represented as an algorithm testing $B(d_N,d_E,d_S,d_W)$ rather than as a subset of $D^{\{N,E,S,W\}}$, and that this algorithmic description is succinct, i.e.\ consumes much less space than $\#D^4$. Then the $N\times N$-representations of $B'$ could consist of a universal Turing machine with some binary encoding of $D$ on its four sides, and a data area containing the program computing $B$. This universal Turing machine could even perform some side calculations for us in its spare time.

The construction may then be iterated: $B'$ may be simulated by $B''$, using $N'\times N'$ grids, etc. Furthermore, all these tilesets may be assumed to contain the same universal Turing machine, and \emph{execute the same program}: a part of the program should ensure that the data area on the simulated tile contains the same program as the original tile.

This is, in effect, a concrete realisation of Kleene's fixed point theorem, that asserts that for every computable map $\pi\colon\{\text{programs}\}\righttoleftarrow$ there exists a program $p$ such that $p$ and $\pi(p)$ are equivalent.

For our purposes, the side calculations to be performed by the universal Turing machine are to enforce a given effective subshift $X\subseteq C^\hor$. Recall from the previous section that the original tileset contains a `$C$' field to store, on each column, a co\"ordinate of $C^\hor$. Macrotiles have access to larger and larger segments of co\"ordinates, namely longer and longer subwords of elements of $C^\hor$, and can test them against a $\Pi^0_1$ condition.

The discussion in~\S\S\ref{sec:Basic},\ref{sec:SingleBit} is very similar to the one in~\cite{durand-romashchenko-shen:1d2d}*{Sections~3--5}, though we make some slightly different choices already at a high level. In~\S\ref{sec:InfinitelyMany} we deal with effective systems rather than subshifts (namely $C$ is the Cantor set $\{0,1\}^\omega$ rather than a finite set; its bits are fed one after the other to the tileset in a ``Toeplitz'' manner). This requires some new ideas (although compared to many other existing fixed-point constructions, including ones found in \cite{durand-romashchenko-shen:fpa08}, they are not particularly difficult). For the ``standard part'' of the discussion, \S\S\ref{sec:Basic},\ref{sec:SingleBit}, we give in the appendix (Section~\ref{sec:Implementation}) a more complete argument with ``explicit UTM details''.

\subsection{The basic construction}\label{sec:Basic}
Let $n_1, n_2, \dots$ be a sequence of numbers; we will describe a tile set for each \emph{level} $k = 1, 2,\dots$. Ultimately, we will only implement the tile set for $k = 1$, and the $k$th tile set appears only through $k-1$ layers of simulation, but it is useful to at first think of all these tile sets as of the same nature, with $k$ just acting as a parameter.

Initially, we have our tiles of level $k$ remember a \emph{position} in $\{0,\dots,2^{n_k}-1\}^2$. Recall that our tilesets have access to finitely many bits of their address, coming from the substitution (e.g.\ the period-doubling substitution from Example~\ref{ex:pd} on both $\hor$ and $\ver$) using which they are built. We synchronize our tiles with the lowest level of the substitution, so that our tilings split into macrotiles. Furthermore, these tiles will share a symbol \texttt{sub}, and we check that the base substitution agrees with \texttt{sub}, in the sense that the contents of the block indeed are a $n_k$th substitutive image of this symbol. Note that at this point, our tiling system admits a unique tiling assuming the base substitution is $2$-stepwise recognizable (every $2 \times 2$-block has a unique desubstitution, by Lemma~\ref{lem:pd}). We will also need to synchronize a symbol \texttt{psub}, which represents the symbol on the next level (its ``correctness'' will be checked later).

Next, we overlay, on top of our tiles, ``wires'' that allow transmitting information between neighbouring macrotiles. For this we add a new layer to our tiles. On the south borders of our macrotiles, the westmost $4 t_k$ tiles will carry a bit on this layer, and using local rules we transmit the first (respectively second, third, fourth) $t_k$ bits to the corresponding area in the neighbouring macrotile on the west (respectively north, east, south). The drawing of the wires can be essentially arbitrary, the only requirement is that it is computationally easy to describe as a function of $k$.

The value of $t_k$ should be large enough to allow sending all the color information. We need $\mathcal O(n_k)$ bits for the colors to send the position, and it turns out that all the other information (\texttt{sub}, \texttt{psub}, \dots) requires very few bits in comparison, so $\Theta(n_k)$ will suffice for everything.

Next, on yet another layer of the SFT we overlay a computation of a universal computational device $M$, in the sense that successive rows, starting from the bottommost one, will contain successive computational steps of $M$. The idea is that the computation of $M$ makes our macrotiles simulate the macrotiles of level $k+1$ (including the computation of $M$, of course). Then our macrotiles will, as desired, merge into macrotiles of level $k+1$, which then merge into macrotiles of level $k+2$, and so on. After all is said and done, we will set $k=1$ and obtain in this manner the desired SFT.

In the implementation section, we will use a ``universal tile set'' as $M$ (as it is easier to work with), but a more standard choice is to use a universal Turing machine. A universal Turing machine takes as input a word of the form $\texttt\%p\texttt\#w\texttt\$$, where the word $p$ describes an arbitrary computational procedure (another Turing machine), which will then be simulated on the input word $w$. Furthermore, the simulation should not take an excessive amount of time. There exist UTM for which the simulation always takes polynomial time in the computational complexity $t$ of $p$ (our universal tile set performs the simulation in time $t \log t$). As stated, the input $\texttt\%p\texttt\#w\texttt\$$ will be written on the south row of the macrotile, on a new layer, and $M$ performs its computation above this.

After the wire contents, we have many other pieces of information that need to be stored on the bottom row, and we make canonical choices about what is stored where: some area contains a number $k + 1$ telling us the level of the tile set; some area of length $\mathcal O(n_{k+1})$ is used for the position of the higher level macro tile; for macrotiles that are part of wires we need to have one bit available for wire transmission; we need a constant number of bits for simulating $M$ itself; and of course the program $p$ has to be written somewhere as well. Finally, all the information is synchronized with the neighbours. Note that for this to be possible, we need the growth constraint that $2^{n_k}$ is much larger than $n_{k+1}$.

\begin{remark}
  The approach we take here is to describe what a single tile needs to remember, and we leave it somewhat implicit how the wires are used to synchronize that information. Of course, the colors of Wang tiles (the wire contents) are the only thing that matters, so it would be more efficient to have nothing but wires on the bottom rows. However, all of the information that is passed around (position, Turing machine state, substitution symbols, bits) is really information about individual cells, so it is more natural to think about the cell information.
\end{remark}

The above description is straightforward to turn into a program for $M$: In all cells, we perform the calculations to check whether we are on a wire, and transmit bits accordingly, and we check that the parent macrotile is simulating $M$ correctly. We also synchronize the level $k$ (i.e.\ if the present cell is supposed to code a bit of this level number, we check that the number is indeed incremented by one) and check the value of \texttt{psub} (by checking that \texttt{sub} is indeed what appears in the present position, when \texttt{psub} is substituted $n_k$ times). Note that $M$ is fixed, so simulating it amounts to part of $p$ describing a simulation of $M$ itself.

On the bottom row, there is a little more to do. What $M$ should do (i.e.\ what $p$ should require it to do) in the level-$k$ tile is to read the current horizontal position of the present tile in the parent (level-$(k+1)$) tile, and depending on this position check that the colours sent northward correctly initialize the computation. If we are in the program area (describing $p$), we should send the corresponding position of $p$. To do this, $M$ will read the corresponding symbol of $p$ from the tape.

\begin{remark}
  This point tends to seem paradoxical on the first sight, but there is nothing particularly tricky about this. If $M$ is a universal Turing machine, the way $M$ typically simulates the machine described by $p$ is that it at all times marks a position on $p$ to remember the current state (a program marker), and it moves back and forth between this position and another marked position on the actual input $w$ (a data marker), modifying only $w$. Nothing goes wrong if we allow the latter marked position to step back inside $p$, although we then do need to make sure that the machine $M$ knows which way it should go, since now the data marker can be to the left of the program marker.
\end{remark}

The other information is initialized similarly. For example, if the co\"ordinates described by the wire bits are in the area describing the current macrotile level, we should check that a bit of the number $k+1$ is indeed written there. If we are in the area containing a bit of \texttt{sub}, we should check that it contains the correct bit of \texttt{psub}.

We now obtain an infinite sequence of simulations. In particular, note that, assuming unique recognizability of the substitution, all the symbols \texttt{sub} are correct.

\subsection{Remembering sequences of bits}\label{sec:SingleBit}
Let us now explain how to remember finite subwords from an initial sequence overlaid on top of our configuration. The sequence is vertically constant, and horizontally can be arbitrary. First, we decide for each macrotile a set of columns that fit inside it, and which it is ``responsible for''. The important things are that this choice is natural (so that we can describe it quickly with a Turing machine), that it does not grow too quickly, and that it does tend to infinity. Our choice is that a macrotile which is at height $h$ in its parent tile (note that we can read its position from the wire contents) is responsible for the columns inside it, which belong in the natural block of size $2^k$ containing the position $h$. One could of course also just take the $h$th such block, but this is (ever so slightly) more complicated to implement.

To make sure that some block is indeed considered by a macrotile, we should have $2^{n_k} > \prod_{i < k}2^{n_i}$, or equivalently $n_k > \sum_{i < k} n_i$. The choice $n_k = 4^{k+c}$ works for this. Note that $n_{k+1} = 4^{k+1+c}$ is certainly $o(2^{4^{k+c}})$, so the growth constraint discussed in the previous section is satisfied. The macrotile will know the word it is responsible for, in a variable \texttt{word}. Similarly as with the substitution symbol, we should also also know the word of the parent \texttt{pword}, so we can check its value is consistent with the actual word the parent is responsible for. Of course, to check such information, the tiles on level $k$ must know the height at which the parent is in its grandparent of level $k+2$, in a variable \texttt{ppos}. It is easy to check that with our choice of the sequence $(n_k)_k$, we have plenty of space for this information.

The check of consistency of \texttt{pword} is easy: We check if the column the present macrotile is on -- when translated to the internal co\"ordinates of the parent tile, are contained in the interval (which is exactly twice longer) that the parent considers. This calculation is straightforward, as we have access to our own position (in the wires), know the position of the parent in the grandparent, and we have simple formulas for all the sizes involved and the choice of the column a macrotile of a particular level is responsible for.

Note that, as we have added new information to our tiles, we need to ensure its consistency. So again macrotiles that sit on the bottom of their parent tile, if they are in a position describing a bit of \texttt{word}, we check consistency with the word \texttt{pword}, and similarly for \texttt{ppos} (checked against wire contents).

Now, consider the tile set obtained by setting $k = 1$ initially, and initially having the correct bit pairs on the tapes of the macrotiles of level $1$ (pairs because $2^1 = 1$) ensured by the SFT rule. An easy induction shows that on every level, on every level the word in the \texttt{word} variable of a macrotile of level $k$ is indeed the word it believes it is. Finally, we can use a little bit of time to check an arbitrary constraint on this word. 

\subsection{Remembering infinitely many bits}\label{sec:InfinitelyMany}

Finally, let us consider effective $\hor$-systems, namely $X\subseteq C^\hor$ with $C=\{0,1\}^\omega$. For each sequence $b\in C^\hor$ to be checked for inclusion in $X$, each co\"ordinate $b_g$ is a sequence $(b_{g,0},b_{g,1},\dots)$ of bits.

Instead of vertically constant sequences, we store each $b_g$ on a column as a Toeplitz sequence: at height $i=2^j(2k+1)$ we store $b_{g,j}$. In other words, in all natural $2$-blocks, the top symbol is $b_{g,0}$. Then out of the bottom bits, the ones that are in the top half of natural $4$-blocks are all $b_{g,1}$, and we continue recursively. The contents of different columns are independent (apart for sharing the underlying substitutive structure). This is easy to achieve with the substitution machinery built in Section~\ref{sec:SubstitutionConstructions}.

We show in this section that we can in a sense remember all these bits, simultaneously on every column. We will then combine this with ideas used in the previous section to get access to multiple columns at once, which then allows performing arbitrary computation on the entire set of sequences.

To do this, we argue similarly as we did with \texttt{word} and \texttt{pword} in the previous section. Now, each macrotile is responsible for a single column \texttt{col} and contains all the bits in that column of the macrotile in a list called \texttt{bits}, and we also know the corresponding information \texttt{pcol}, \texttt{pbits}. We need to remember exactly $\ell_i\coloneqq\sum_{i \leq k} n_i$ bits in \texttt{bits}, where $L_k\coloneqq2^{\ell_i}$ is the size of the $k$-level macrotile (in absolute number of cells per side, assuming we are actually dealing with the first level tile set). The first $\ell_k-1$ bits are exactly $b_0 b_1 ... b_{\ell_k-2}$, and the last one -- the bit on the southmost row -- can in principle represent any $b_i$ with $i\geq\ell_k-1$. We call this the \emph{high-level bit}.

We need to make a minor modification to how we pick the column we are responsible for. Namely, previously the macrotile of level $k$ at height $h$ in its parent tile was responsible for the bits in the (word that the) column $h$ (belonged to), and we did not have a separate \texttt{col} variable. We cannot do the same now: Suppose a macrotile of level $k$ is responsible for the $j$th column. Then in particular, it needs to know the bit in this column on the bottom row (the high-level bit $b_i$). This bit is of course known to a \emph{single} macrotile of level $k-1$ inside it, namely the one on the bottom row. Thus, on the bottom row, the $(k-1)$-macrotile containing the $j$th column of its parent tile must itself be responsible for that column.

To combat this, at all ``key heights'', namely whenever \texttt{ypos} is zero or a power of $2$, the column of responsibility is changed to be the $j$th in all tiles. As we explain later, it is important (for getting an almost $1$-to-$1$ extension) that at other heights the column of responsibility is chosen based on the height only, so on these columns we set $\texttt{col} \equiv \texttt{ypos} \bmod L_{k-1}$. By our choice of $n_k$, note that $\texttt{ypos}\bmod L_{k-1}$ does reach every value at non-special heights (for example, in the last quarter of the heights). Just as with \texttt{word} and \texttt{pword}, the information between a child and a parent's bit information is synchronized whenever the column of responsibility is the same.

A bit of intelligence is needed to do this synchronization, namely we should check which bit we are responsible for, using the Toeplitz structure: The tile at the key height $0$ will synchronize its high-level bit with the high-level bit of its parent. The tile at key height $2^h$ will synchronize its high-level bit with the $(h + \ell_{k-1})$th bit of \texttt{pbits}. At non-key heights, we should look at the position of the last $1$ in the binary representation of \texttt{col}  and then use the same formula to figure out which bit we are responsible for.

Next, now that each tile knows all the bits of one column inside it, it is not difficult to adapt the construction of the previous section to propagate words instead of individual symbols. We can mostly forget the details of the previous construction, except the fact that we remember the bits coded on a particular column.

As in the previous section, let us again have each tile be responsible for the contents of a sequence of columns inside it, but now it is a word of length $2^k$ over the alphabet $\{0,1\}^k$, where a single symbol represents the bits $b_0, ..., b_{k-1}$ of a particular column. We call this area the \emph{stripe of responsibility}. We want the level to ``drag behind'' the \texttt{bits} process, so we take the stripe of responsibility to be simply the stripe that contains the \texttt{pypos}th column of the grandparent, if it is inside the present tile, and otherwise we do not have a stripe of responsibility. Let us call the information \texttt{psword} (for stripe word of the parent).

To ensure that the bits of \texttt{psword} are correct, we use \texttt{bits}: we observe that the growth of the depth up to which we are responsible for on the stripe of responsibility is small enough that every tile of level $k-1$ already knows all bits up to this level on its column of responsibility (without looking at the high-level bit). Thus, any tile that is responsible for the column containing the grandparent's \texttt{pypos} column (and one exists if said column intersects the parent tile at all, since \texttt{col} goes through every column) can actually supply these bits.

\begin{remark}
  In \cite{durand-romashchenko-shen:fpa08}, a similar-looking ``two-step'' approach was already used for remembering single bits, i.e.\ tiles were first forced to correctly know a single bit, and then these were collected into words in a separate step. We did not take this approach in the previous section, as it does not seem necessary there, and despite the similarity, the flow of information here is slightly different.
\end{remark}

Finally, having access to the stripe words, we can easily verify an arbitrary condition on them.

\begin{remark}
  This construction can be readily adapted to the $\Z^2$-case as well, i.e.\ we can read an infinite sequence of bits into macrotiles if they are already encoded into Toeplitz sequences. For completeness, we write the precise statement that one obtains on the plane (whose proof is exactly the same as discussed in this section).

Define $T \subset \{0,1\}^\Z$ as the subshift where for each $x \in T$, letting $y^i_n = x_{2n + i}$, one of the configurations $y^i$ is constant, and the other is in $T$. Then there is a shift-invariant continuous map $T \mapsto \{0,1\}^\N$ that reads the bits on the constant sequences, and the pull-back of an effectively closed set $C \subset \{0,1\}^\N$ is an effectively closed subshift. If $T' \subset T$, we say \emph{the skeleta of $T'$ agree} if we can pick $i \in \{0,1\}$ so that the configurations $(x_{2n+i})_n$ are constant for all $x \in T'$, and the skeleta of $\{(x_{2n+i+1})_n\mid x \in T'\}$ agree.

\begin{thm}
\label{thm:InfinitelyMany}
Suppose $X$ is the (non-sofic!) $\Z^2$-subshift where each column is independently taken from $T$, so that the skeleta of columns agree. Then for every effective $\Z$-system $Z$ there exists a sofic shift $Y$ such that the horizontal subaction of $X \cap Y$ factors onto $Z$.
\end{thm}
\end{remark}

\subsection[The SFT $\wee$ on $\lamp$]{\boldmath The SFT $\wee$ on $\lamp$}\label{ss:W}

In this section, we introduce and study in detail an SFT $\wee$, which is a very efficient implementation of the idea from Section~\ref{sec:Corooted}. The idea is to run two copies of the SFT with spiders $\Theta_0$, one upward and one downward, and synchronize them, although we elect to introduce this SFT through tetrahedra instead. 
It is by far the most interesting ``small, naturally occurring'' SFT we have found on the lamplighter group, and thus may be of independent interest. However, only our strongest statements actually need the results of this section.

In $\wee$, the vertices of $\lamp$ are coloured by the set $C_0=(\Z/3)^2$, and the allowed tetrahedra are
\[ \Delta_{\mathrm{tree}} = \{((t,u),(t,u),(t,u+1),(t+1,u)) \mid (t,u) \in C_0\} \]
(as usual, in the order $(v, a v, b^{-1}a v, b v)$) for each $(t,u) \in C_0$. This gives a total of $9$ tetrahedra, and as is our convention, we include also the other orientations, to get a total of $18$. More pictographically (the lamps are explained below) this is
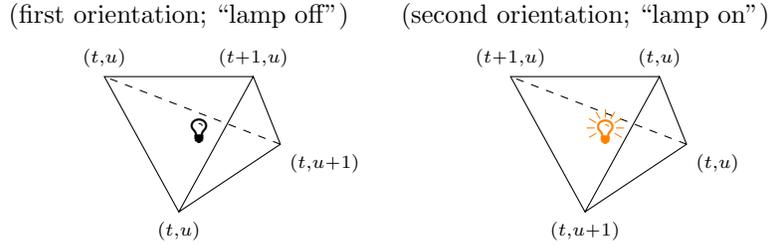
\[  \begin{tikzpicture}[scale=0.9]
    \foreach\label/\o/\a/\ab/\b/\shift/\lamp in {(first orientation; ``lamp off'')/{(t,u)}/{(t,u)}/{(t,u+1)}/{(t+1,u)}/0cm/\large\faLightbulbO,(second orientation; ``lamp on'')/{(t,u+1)}/{(t+1,u)}/{(t,u)}/{(t,u)}/6cm/\faLightbulbON} {
      \begin{scope}[xshift=\shift]
        \coordinate[label={below:$\scriptstyle\o$}] (A) at (0,0);
        \coordinate[label={below right:$\scriptstyle\ab$}] (B) at (1.5,1);
        \coordinate[label={above:$\scriptstyle\a$}] (C) at (-1.1,2);
        \coordinate[label={above:$\scriptstyle\b$}] (D) at (1.1,2);
        \draw (A) -- (B) -- (D) -- (A) -- (C) -- (D);
        \draw[dashed] (B) -- (C);
        \node at (0,2.9) {\label};
        \node at (0.3,1.2) {\lamp};
      \end{scope}
    }
  \end{tikzpicture} \]

In addition to the mental picture of trees, explained in the previous section, another important intuition to keep in mind is that the left component is counting lamps that are on below the head, modulo $3$, and the right component counts the lamps that are on above the head. Since this is a subshift, it cannot know the ``real'' orientation of the lamplighter group, so at each node, either $a$ or $b$ will correspond to flipping a lamp (but as we will see, this orientation will be consistent throughout the group). For example, when moving upwards, i.e.\ from $v$ to $av$ or $bv$, exactly one choice will move upward and leave a lamp lit on the edge it traverses, and this corresponds to the edge of the tetrahedron whose vertices are labeled $(t,u)$, $(t+1,u-1)$ (since one more lamp is now on on the left, and one less on the right). In the figure, we have included a lamp, which indicates whether the light ``at the tetrahedron'' is on, from the perspective the front bottom node $v$.

We will prove the following properties of $\wee$:
\begin{itemize}
\item It has similar structure as the SFT coming from $\Theta_0$, i.e.\ for every $\ver$-coset we can interpret it as the frontier of a corooted binary tree. (Lemma~\ref{lem:WeeTheta})
\item It is nonempty. (Lemma~\ref{lem:WeeNonempty})
\item Every configuration in it is $\sea$-aperiodic. (Lemma~\ref{lem:WeeAperiodic})
\item It is of zero entropy. (Lemma~\ref{lem:WeeZeroEntropy})
\item It is a $(\Z/3)^2$-extension of the lamp action (see~\S\ref{ss:lamplighter}) of the lamplighter group. (Proposition~\ref{prop:WeeTechnical})
\item It is minimal and uniquely ergodic. (Corollary~\ref{cor:Wstrictlyergodic})
\end{itemize}
As a corollary, we obtain also that the action of the lamplighter group on the space of lamps is sofic, but not SFT, see Corollary~\ref{cor:lampactionsofic}.

(Some of these statements imply the others, but this list follows the order in our exposition.)

To begin the study, consider the first co\"ordinate $t$. Projecting to this coordinate, $\wee$ is, in the language of tetrahedra, the same subshift as the tileset $\Theta_0$ above: keeping all edges in $\lamp$ whose $t$-co\"ordinate are equal produces a union of binary trees growing ``downwards'' (in negative $\phi$-direction). More precisely, a direct calculation proves the following lemma (noting that the existence of a block map does not require the subshift to be nonempty).

\begin{lem}
\label{lem:WeeTheta}
The subshift $\wee$ admits a block map into the subshift defined by spiders $\Theta_0$.
\end{lem}

Similarly, the second co\"ordinate $u$ serves to mark binary trees growing ``upwards'', and the combination of both marks two families of binary trees. The additional property is that the removed edges are the same for the upwards and downwards families. As hinted above, a priori this could mean the subshift is empty (in which case the previous lemma would be trivial). We show that it is not empty.

\begin{lem}
\label{lem:WeeNonempty}
The subshift $\wee$ is not empty.
\end{lem}

\begin{proof}
We follow the intuition from the third paragraph of this section, and define
\[y(s,n)=(\#\{k<n:s_k=1\}\bmod 3, \#\{k>n:s_k=0\}\bmod 3). \]
Consider any tetrahedron $(v, av, b^{-1}av, vb)$, and suppose $y(v) = y(s,n) = (t, u)$. A direct calculation shows that if $s_{n+0.5} = 0$, then we see a tetrahedron in the first orientation, and if $s_{n+0.5} = 1$, we see one in the second orientation.
\end{proof}

\begin{lem}
\label{lem:WeeAperiodic}
The subshift $\wee$ is does not contain $\sea$-aperiodic points.
\end{lem}

\begin{proof}
Let $y \in \wee$ be arbitrary. Consider $s \in \sea$, and write $s = hv$ with $h \in \hor, v \in \ver$. By Lemma~\ref{lem:WeeTheta}, the first coordinate stays constant when moving south and north (i.e.\ when we flip lamps above the head, i.e.\ in its right $\ver$-coset). Symmetrically, the second coordinate stays constant when moving west and south. Suppose now $h \neq 1$, the case $v \neq 1$ being symmetric. In this case, let  $x$ be the projection of $y$ to the first coordinate. Then $x = vx$ so it suffices to show that $x$ is not $h$-periodic.

For this, we use the trees given by Lemma~\ref{lem:WeeTheta}: Recall that in any large tetrahedron of any configuration, looking at a height $k$ tetrahedron ``from the front'', we can see it as a binary tree, and since right cosets of $\ver$ have constant colour, there is a natural way to write labels on the vertices. Then by the structure result for the subshift defined by $\Theta_0$, we in fact see a binary tree with $k+1$ levels where, if the bottom node has label $t$, there is a unique \emph{special path} from the bottom node (corresponding to a coset of $\ver_{k+1}$) to some element on the top. Let us call this the \emph{special node} of the tetrahedron (note that it indeed corresponds to a right coset of $\ver_0$, i.e.\ a single element).

Now, taking the tetrahedron of height $k$ below the identity element, we see that any nontrivial $h$-shift will change the special node of any sufficiently large tetrahedron, showing that indeed $x \neq hx$. The case of $v \neq 1$ is completely symmetric, using the symmetric version of the structure result for $\Theta_0$, and using a tetrahedron above the identity element. 
\end{proof}

Aperiodicity in the $\sea$-direction appears to be very common for SFTs on the lamplighter group, in that we found many examples with this property (which did not lend themselves to proving full strong aperiodicity) before finding $\wee$.

\begin{lem}
\label{lem:WeeZeroEntropy}
The subshift $\wee$ has zero entropy.
\end{lem}

\begin{proof}
Consider a height-$k$-tetrahedron above the identity, and project to the first coordinate. After suitable $\hor$-shift permuting the top row $\hor_k v$ (meaning shifting by elements that flip lamps above $k$), we may assume the special path is the path from $1$ to $a^k$, i.e.\ the front left path of the tetrahedron. Symmetrically, in the second coordinate, looking at the tetrahedron from the left, we see a downward oriented binary tree with a single special path from top to bottom, and after a suitable $\ver$-shift we may assume this is the path from $a^k$ to $1$. Furthermore, it is clear that $\wee$ is closed under cellwise $C_0$-translation (i.e.\ summing the same $(t, u) \in C_0$ to the vertex label at every node). 

We claim that there is in fact a unique such tetrahedron for each $(t, u)$. This of course shows that the number of legal tetrahedra is precisely $9 \cdot 2^{2k}$. Since the cardinality of a height $k$ tetrahedron is $(k+1) 2^k$, this shows that the counting entropy per site is
\[ \frac{\log{9 \cdot 2^{2k}}}{(k+1) 2^k} = \frac{\log{9} + 2k \log 2}{(k+1) 2^k} \overset{k \rightarrow \infty}{\longrightarrow} 0, \]
so entropy is zero (and convergence is very fast).

Of course, the unique configuration is be precisely $(t, u) \cdot y$ where $y$ is the configuration from Lemma~\ref{lem:WeeNonempty}. The reader may find it useful to convince themselves that this happens for small $k$ and for $(t,u) = (0,0)$, using the already established properties of $\wee$ that right $\vee$-cosets have constant values in the first coordinate, and right $\hor$-cosets have constant values in the second coordinate (as well as the defining rules).

For the general proof, we proceed by induction. If the special path in both coordinates is the one from $1$ to $a^k$, then in particular the special edge (i.e.\ the one along which the values does not change) is always the $a$-edge, for nodes $a^\ell$. Looking at the tetrahedron above $a$, induction shows that this is precisely the corresponding tetrahedron from $y$. Now look at the tetrahedron of height $k-1$ above $b$. Since the second projection is constant on right $\hor$-cosets, the edges for the second projection on the path from $b$ to $a^{k-1} b$ must be special. Thus, they are also special for the first projection. Thus, the tetrahedron above $b$ is the uniquely determined tetrahedron based on $(t+1, u)$, and by induction this is the height $k-1$ tetrahedron above the origin in $(t+1,u) \cdot y$. This is of course precisely the tetrahedron of height $k-1$ rooted at $b$ in $y$, concluding the inductive proof.
\end{proof}

We now proceed to the main structure theorem for $\wee$. 

\begin{prop}\label{prop:WeeTechnical}
  The SFT $W$ admits a free diagonal action of $C_0$, with quotient $X\coloneqq\{0,1\}^\E$. The action of $\lamp$ on $X$ is induced by the natural action of $\lamp$ on $\E$. In other words, for some cocycle  $\eta\colon\lamp\times X\to C_0$, there is a bijection of $W$ with $C_0\times\{0,1\}^\E$ such that the action of $g\in\lamp$ is given by $(c,x)\mapsto(c+\eta(g,x),g x)$.
\end{prop}
\begin{proof}
  We showed in the previous lemma that the contents of a tetrahedron is completely determined by the special nodes on the top row and bottom column, up to a $C_0$-translation. On the other hand, since the projection to first coordinate is constant on right $\ver$-cosets and the second on right $\hor$-cosets, all the information in a configuration is contained in the ``infinite tetrahedron'' (union of all tetrahedra) above the origin, and the one below the origin. Thus, a single configuration is nothing but a choice of two paths to infinity, one upward (in the positive $\phi$-direction), and one downward, together with the value at the origin. Furthermore, one can obtain any such pair of paths by taking a limit of $\sea$-translates (since the maps extracting the paths are continuous), and then applying a suitable $C_0$-translation.

  From this description, it is straightforward to produce an explicit bijection $\theta$ between $C_0\times X$ and $W\subset C_0^\lamp$ and to compute the cocycle $\eta\colon(C_0\times X)\times\lamp\to C_0$ via this bijection. Given $(c,d)\in C_0$, $x\in X$ and $(s,n)\in\lamp$, the cocycle is defined as follows, with `$[p]$' the expression equal to $1$ if $p$ is true and $0$ if $p$ is false:
  \[\left(c + \sum_{i < n}[s(i)\ne x(i)] - \sum_{i<0} [0\ne x(i)],d + \sum_{i>n}[s(i)\ne x(i)] - \sum_{i>0}[0\ne x(i)]\right);\]
note that, on both terms, the sums are infinite but all terms cancel except finitely many so they reduce to finite sums.
    
  It is easy to check that the image of $\theta$ belongs to $W$, by computing the values on each tetrahedron. This map is obviously injective, since the symbol at the origin is $(c, d)$, the special top nodes of tetrahedra above the identity determine the positive values of $x$, and the special bottom nodes of tetrahedra below the identity determine the negative values of $x$. Similarly (and by the first paragraph) it is surjective. Finally, it is obviously continuous, so it is a homeomorphism between the systems. 

  Using $\theta$, it is easy to compute the cocycle as $\eta(g,x)=\theta((0,0),x)(g)-\theta((0,0),x)(1)$. Explicitly, the cocycle $\eta$ is given by
  \begin{equation}\label{eq:cocycle}
    \eta(a,x)=(x(\tfrac12),-x(\tfrac12)),\qquad\eta(b,x)=(1 - x(\tfrac12), -x(\tfrac12))
  \end{equation}
  (extended multiplicatively: $\eta(1,x)=0$ and $\eta(g h,x)=\eta(g,h x)+\eta(h,x)$).
\end{proof}

Intuitively, the cocycle simply describes how the ``sum of lamps'' changes under action of $g$, on both sides of the street. The cocycle value is
\[ (\text{number of differences on left}, \text{number of differences on right}). \]

In the following statement, recall that a subshift is \emph{strictly ergodic} if it is minimal and uniquely ergodic.

\begin{cor}\label{cor:Wstrictlyergodic}
  $W$ is a strictly ergodic SFT. 
\end{cor}

\begin{proof}
  By definition, $W$ is an SFT. To prove that $W$ is strictly ergodic, we show that in large tetrahedra every pattern occurs asymptotically with the same frequency. Now the factor $X$ is a (compact abelian) topological group, and admits a unique normalized Haar measure $\mu$, i.e.\ a unique probability measure that is invariant under translations. The subgroup $\sea$ acts by translations, thus preserves $\mu$. Conversely since $\sea$ is dense in $X$ any measure preserved by $\sea$ is preserved by general $X$-translations, thus $\mu$ is the unique measure preserved by $\sea$. The action of $\lamp$ is by automorphisms of $X$, thus also preserves $\mu$. This shows unique ergodicity of the action of $\lamp$ on $X$. Minimality is clear.
  
  It remains to see that the cocycle values are equidistributed. Consider a large tetrahedron of size $2n$, contained in an even larger tetrahedron of size $2(n+m)$, and a vertex $v$ at mid-height; by~\eqref{eq:cocycle} the value in $C_0$ at $v$ depends non-trivially on all the values of $x$ in the outer regions of height $m$, without affecting the values of $x$ in the inner region of height $2n$. Thus for large $m$ the value in $C_0$ is arbitrarily well randomized. 
\end{proof}

\begin{cor}\label{cor:lampactionsofic}
  The natural action of $\lamp$ on $\{0,1\}^\E$ is a sofic shift but not an SFT.
\end{cor}
\begin{proof}
  The shift $X=\{0,1\}^\E$ is by construction a quotient of an SFT, so it suffices to prove that $X$ is expansive. This is an immediate consequence of the fact that $X$ is the quotient of a subshift by a free shift-commuting action of a finite group, to wit $C_0$; or explicitly, we see that already the action of the subgroup $\langle a \rangle$ is expansive, indeed it is the binary full shift.
  
  To prove that $X$ is not an SFT, we exhibit \emph{pseudo-orbits that are not shadowed by orbits}. Recall that an $\epsilon$-pseudo-orbit in the action of a group $G=\langle S\rangle$ on a metric space $X$ is a map $\phi\colon G\to X$ with $d(\phi(s g),s\phi(g))<\epsilon$ for all $s\in S,g\in G$, and that it $\delta$-shadows an orbit $(g x)_{g\in X}$ if $d(g x,\phi(g))<\delta$ for all $g\in G$. Now SFTs have the \emph{shadowing} property~\cite{walters:pseudo-orbit}: for every $\delta>0$ there is $\epsilon>0$ such that every $\epsilon$-pseudo-orbit $\delta$-shadows an orbit.

  Now given $\epsilon>0$, choose $k\in\N$ such that any two configurations in $X$ that agree on $[-k,k]$ are at distance at most $\epsilon$. Define a $\delta$-pseudo-orbit as follows: consider first the natural action of $\lamp$ on $(\Z/2)^{\Z/2k}$, namely lamps arranged along a circular street, with $a$ acting by rotation and $b$ acting by ``rotation while flipping the encountered lamp''; lift this action to $X$, so $a$ acts by shifting while $b$ acts by shifting and toggling all lamps at position $\frac12+2k i$ for all $i\in\Z$. This is clearly an $\epsilon$-pseudo-orbit; but it is not $\delta$-shadowed by an orbit as soon as $\delta$ is smaller than the distance between (say) $0^\E$ and $1^\E$.
\end{proof}

\section{Proofs of the main statements}
The constructions from the previous section are all that is needed to complete the proofs of our main statements, which we recall for convenience:

\begin{thm}[$\supseteq$ Theorem~\ref{thm:effectiveH}]\label{thm:effectiveH2}
  Let $X$ be an effective $\hor$-system. Then $\iota_*\eta^*(X)$ admits an SFT cover. The SFT cover and covering block map are effectively computable from $X$.

  Furthermore, writing $Y$ for the lamp action of the lamplighter group, 
  the system $\iota_*\eta^*(X) \times Y$ is almost $(\Z/3)^2$-to-$1$ (again with all the data effectively computable). In particular, in this case $\iota_*\eta^*(X)$ always admits a SFT cover of the same entropy (which is precisely the entropy of $X$).
\end{thm}

\begin{proof}
  We first discuss the case of a binary subshift $X\subseteq\{0,1\}^\hor$, and then explain how to deal with the general case, and how to minimize the size of the extension (in the sense of the second paragraph of the statement). We select, for definiteness, the period-doubling substitution $\tau$ from Example~\ref{ex:pd}.
  
  We can use the constructions of Section~\ref{sec:SubstitutionConstructions} to obtain a subshift where on the $2^n$-by-$2^n$ blocks of each $\sea$-coset, up to normalization we always see the pattern $(\tau\times\tau)^n(a)$ for $a \in \{0,1\}$, and on another layer we see sequences that are constant on right $\ver$-cosets. We construct a classical Wang tile set which realizes the fixed point construction in valid Wang towers, and with additional computation checks that the configuration on the rows is one of $X$ (in the sense that it enumerates some number of forbidden patterns, and checks that they do not appear, and as $n \rightarrow \infty$ it deals with all patterns). Then we use Section~\ref{ss:wang2sofic} to actually implement this on the lamplighter group (again using the constructions from Section~\ref{sec:SubstitutionConstructions}).

  For the case of general $X$, we use constructions from Section~\ref{sec:SubstitutionConstructions} to ensure that all columns carry a Toeplitz configuration with the same branching structure on each coset, and then we use the construction from Section~\ref{sec:InfinitelyMany} to remember the infinitely many bits on each column. Again, eventually we have enough space to check any condition on these bits.

 For the second paragraph, i.e.\ optimization of the extension, we consider $\iota_*\eta^*(X) \times W$. We show that this system is AFT. Then, since $W$ is a $(\Z/3)^2$-to-$1$ cover of $Y$, by definition, $\iota_*\eta^*(X) \times W$ is an almost $(\Z/3)^2$-to-$1$ cover of $\iota_*\eta^*(X) \times Y$.
 
  To show the AFT claim, we note that since $W$ admits a block map into $\Theta_0$ and its reverse, we have trees at our disposal. We can use these trees to implement everything we needed from Section~\ref{sec:SubstitutionConstructions}, i.e.\ we can use any configuration of $W$ as the ``skeleton'' on which substitutions are implemented. 
  
  The verification that the first layer contains an element of $\iota_*\eta^*(X)$ can be performed no matter how the substitutive structure plays out, so we need not allow an additional branching sequence, and can directly use a substitution along the $W$-trees, and send all information over them. 

 Then given an element of $W$, and knowing an element of the simulated $\hor$-system $X$, we can deduce the entire configuration of the SFT on $\lamp$, with probability one with respect to any invariant measure. Namely, since $\tau$ is almost odometric, with probability one the $\tau$-substitutive configuration is uniquely determined by the $Y$-configuration, in particular it is determined by the $W$-configuration. Similarly, the Toeplitz configurations used to implement the infinite data (when $X$ is not a subshift) have no extra bits with probability one.

 The only data we do not yet know is the precise contents of the macrotiles. However, the data in a macrotile at level $k$ is uniquely determined unless it is carrying some information about higher level tiles. Clearly this happens for less than half of the macrotiles on a particular level, since our construction of macrotiles is completely deterministic (or, more precisely, unambiguous). This depends only on the corresponding odometer point and the odometer is uniquely ergodic, so with probability one every macrotile is part of a bigger one that carries no such information.
 
 Finally, for the last claim, simply note that $W$ has zero entropy and that an AFT has the same entropy as its SFT cover.
\end{proof}

\begin{remark}\label{rem:CannotBe}
  The result cannot be strengthened to the claim that $\iota_*\eta^*(X)$ is AFT. Indeed, consider for $X$ the unique non-transitive two-point $\hor$-system (it is clearly effective, but not of finite type). The $\lamp$-subshift $Y\coloneqq\iota_*\eta^*(X)$ coincides with the pull-back $\phi^*(\{0,1\}^\Z)$ of the binary full shift. Assume for contradiction that $Y$ is AFT, so there is an SFT cover $W\twoheadrightarrow Y$ which is generically $1:1$. Without loss of generality, $W$ is nearest neighbour and the coding map $W\subseteq A^\lamp\to Y\subset\{0,1\}^\lamp$ is letter-to-letter. Consider a generic $y\in Y$; by genericity there is a unique $w\in W$ above it. Now $w$ is constant on each coset of the sea-level, so in particular on every $\ver$-coset. We can then copy the labels from $w$ above some $\ver$-coset, and the portion of $\lamp$'s Cayley graph above it, into a position at a different height in $\lamp$ whose $\ver$-coset has the same label; since generically $z$ is not periodic, this creates a configuration in $W$ which is not constant on a sea-level, a contradiction.
\end{remark}

\begin{thm}[= Theorem~\ref{thm:effectiveZ}]\label{thm:effectiveZ2}
  Let $X$ be an effective $\Z$-system. Then $\phi^*(X)$ admits an SFT cover with a computable covering map, both effectively computed from $X$.
\end{thm}

\begin{proof}
  Let $f\colon X \to X$ be the homeomorphism generating the effective $\Z$-system $X$. Pick any injective and computable coding map $\xi\colon X \to \{0,1\}^\hor$ such that the $i$th bit of $x \in X$ can be decoded from  any large enough $\hor_k$-coset. The subshift $Y \subset (\{0,1\}^\hor)^2$ we consider is essentially an encoding of the graph of $f$: we have $(y, y') \in Y$ if and only if $y$ encodes a point $x\in X$ and the shift of
  $y'$ encodes $f(x)$, in formulas $y=\xi(x)$ and $y'_h=\xi(f(x))_{\sigma(h)}$ for all $h\in\hor$. Clearly $Y$ is an effective subshift, so $\iota_*\eta^*(Y)$ is sofic.

Now we construct an SFT $Z$ on $\lamp$ by adding to the SFT cover of $Y$ the rule that if $z\restriction_\hor \mapsto (y, y')$ and $h \in\hor_0$ then $(a^{-1}z)\restriction_\hor$'s first component is $y'\circ\sigma^{-1}$ (which is well-defined since $y'$ is invariant under $h_0$). This synchronizes the contents of $\sea$-cosets so that the $a$-shift of $z$ will transform the configuration $x$ coded on $\sea$ to $f(x)$; in other words, the map $f$ on $X$ is implemented by the $a$-shift on $Z$.
\end{proof}

\begin{cor}[= Corollary~\ref{cor:examples}(1)]\label{cor:examples2}
  There exists a strongly aperiodic SFT on the lamplighter group.
\end{cor}
\begin{proof}
  Let $X\subseteq A^\hor$ be any effective subshift without periodic points, and let $Y \subseteq A^\Z$ be an effective subshift without periodic points. Then $(\iota^-)_*(X) \times \phi^*(Y)$ is sofic and aperiodic.
\end{proof}

Recall that $\Pi^0_1$ is the complexity class of problems that can be (computably, many-one) reduced to the complement of the halting problem of Turing machines. A problem is $\Pi^0_1$-\emph{hard} if the converse holds, and $\Pi^0_1$-\emph{complete} if it is $\Pi^0_1$ and $\Pi^0_1$-hard; see e.g.~\cite{rogers:computability}.

\begin{cor}[= Corollary~\ref{cor:undecidable}]\label{cor:undecidable2}
  The tiling problem on the lamplighter group is undecidable, and more precisely is $\Pi^0_1$-complete.
\end{cor}
\begin{proof}
  For an effective $\hor$-subshift $X$, its emptiness problem is by definition $\Pi^0_1$-complete. The tiling problem on the lamplighter group is, again by definition, the emptiness problem for SFTs on $\lamp$, which reduce by Theorem~\ref{thm:effectiveH} to the previous problem, so it is $\Pi^0_1$-hard. Conversely, the problem is clearly in $\Pi^0_1$.
\end{proof}

\begin{cor}[= Corollary~\ref{cor:examples}(2)]\label{cor:examples3}
The lamplighter group admits an SFT where no configuration is recursively enumerable.\qed
\end{cor}

Recall that the Kolmogorov complexity of a string $s\in\{0,1\}^n$ is the minimal length of a Turing machine computing $s$. Naturally $\{0,1\}$ may be replaced by any finite alphabet $A$, and $n$ may be replaced by any finite set, for example the subset $\bigcup_{k=0}^n\{a,b\}^k$ of $\lamp$.
\begin{cor}[= Corollary~\ref{cor:examples}(3)]\label{cor:examples4}
The lamplighter group admits a non-empty SFT such that the Kolmogorov complexity of every pattern visible on a radius-$n$ ball in $\lamp$ is $2^{\Theta(n)}$.\qed
\end{cor}

\begin{cor}[= Corollary~\ref{cor:examples}(4)]\label{cor:examples5}
The entropies of SFTs on the lamplighter group are precisely the upper semi-computable (namely, $\Pi^0_1$) nonnegative real numbers.
\end{cor}
\begin{proof}
  Every $\Pi^0_1$ nonnegative real number may be written as $d\log t$ for some $t\in\N$ and $d\in[0,1]$, where $d$ is also $\Pi^0_1$. There exists an effective Sturmian $\Z$-subshift $X$ of density $d$, namely $X\subseteq\{0,1\}^\Z$ such that every $x\in X$ satisfies $\sum_{i=1}^n x(i)\approx d n$. By Theorem~\ref{thm:effectiveZ}, there exists an SFT cover $Y\subseteq A^\lamp$ of $\phi^*(X)$ via $\pi\colon A\to\{0,1\}$. Consider now the subshift
  \[Z=\{(y,s)\in Y\times\{0,1,\dots,t\}^\lamp\mid \pi(y(g))=0\Leftrightarrow s(g)=0\},\]
  namely for every $x\in X$ allow arbitrary values in $\{1,\dots,t\}$ at levels marked by $1$ in $x$, and only $0$ at levels marked by $0$. This is clearly still an SFT. Consider the projection $\rho\colon Z\to\{0,1,\dots,t\}^\lamp$ on the second component.

  On a large ball, or more simply on a height-$n$ tetrahedron, we have $(n+1)2^n$ vertices, and see $\approx t^{d(n+1)2^n}$ different patterns. The entropy of $\rho(Z)$ is therefore $d\log t$.

  Finally, the extension $\rho$ does not add any entropy by the last sentence of Theorem~\ref{thm:effectiveH2}. 
\end{proof}

The following theorem is not new, but we obtain a new proof of it, illustrating the generality of simulation methods. Recall that the \emph{sunny-side-up} shift on a group $G$ is the set of configurations $x\in\{0,1\}^G$ with at most one `$1$'.

\begin{thm}
  The sunny-side-up is sofic on the lamplighter group.
\end{thm}
\begin{proof}
  The sunny-side-up on $\hor$ is an effective subshift, so the induction of its pull-back to $\sea$ is sofic by Theorem~\ref{thm:effectiveH}, say a factor of an SFT $X\subseteq A^\lamp$ by a map $\pi_1\colon A\to\{0,1\}$. Symmetrically, the sunny-side-up on $\ver$ is an effective subshift, so the induction of its pull-back to $\sea$ is sofic by Theorem~\ref{thm:effectiveH}, say a factor of an SFT $Y\subseteq B^\lamp$ by a map $\pi_2 \colon B\to\{0,1\}$. Finally, the sunny-side-up on $\Z$ is sofic, so its pull-back to $\lamp$ is also sofic, say a factor of an SFT $Z \subseteq C^\lamp$ by a map $\rho \colon C\to\{0,1\}$. Set $W =X\times Y \times Z \subseteq(A\times B \times C)^\lamp$ with factor map $\sigma(x,y,z)=\min\{\pi_1(x),\pi_2(y),\rho(z)\}$. This defines an SFT cover of the sunny-side-up on $\lamp$.
\end{proof}

We finally prove that substitutive $\sea$-subshifts are sofic. This is analogous to what happens on the plane. We explain how to obtain this from the fixed point argument, only sketching the differences with the proof of Theorem~\ref{thm:effectiveH}.
\begin{thm}[= Theorem~\ref{thm:substitutive}]\label{thm:substitutive2}
Let $X$ be a substitutive $\sea$-shift. Then $\iota_*(X)$ is sofic.
\end{thm}
\begin{proof}
In the proof of Theorem~\ref{thm:effectiveH}, we already keep track of a point in the substitution $\tau \times \tau$, with $\tau\colon A \to A^2$. The Turing machine sweeps once over the fields storing the $x$ and $y$-co\"ordinate, to compute the image of a symbol in this substitution, and we can effortlessly combine $\tau$ with any deterministic substitution $\tau'\colon A \to A^{2\times2}$ defining $X$.

If $X$ is defined by a non-deterministic substitution $\tau'$, we need to be a bit more careful, making sure that the non-deterministic choices are consistent. Perform the basic fixed point construction in which macrotiles of level $k$ always consist of $N\times N$ macrotiles of level $k-1$, for a fixed $N = 2^\ell$. It suffices to evaluate the $\ell$th power of $\tau'$. For this, we assume that each macrotile of level $k$ knows a word -- the \emph{choice sequence} -- of length $\ell$ (over an alphabet of cardinality at most $\#A$), which determines all the non-deterministic choices leading to it. The information must be synchronized between the children, in that the first $j$ symbols of the word have to be shared by macrotiles that are in the same basic $2^j$-block of macrotiles.

It is easy to figure out, from the binary digits of the $x$ and $y$ co\"ordinates, which subword of the choice sequence has to be synchronized with which neighbour. Each macrotile then just sends a word of length at most $\ell$ to each of its neighbours. This is easy to achieve; note that already sharing the co\"ordinates of a macrotile in the parent tile requires sharing words of this length. Once we have the choice sequence, we can proceed as in the first paragraph of the proof. We omit further details.
\end{proof}

We finally rule out in a quite general setting a possible strengthening of Theorem~\ref{thm:effectiveH}:
\begin{thm}[= Observation~\ref{obs:effectivenonsofic}]\label{obs:effectivenonsofic2}
  Let $\iota\colon H\hookrightarrow G$ be a subgroup inclusion, and assume that $G$ is amenable and $H$ is infinite and has decidable word problem. Then there exists an effective $H$-subshift $X$ such that $\iota_*(X)$ is not sofic.
\end{thm}
\begin{proof}
  We consider the ``copy'' shift on $H$: its alphabet is $A=(\{\circ,\bullet\}\times\{0,1\})^{\{\mathsf{in},\mathsf{out}\}}$ with components respectively called \textsf{in\_mark}, \textsf{in\_data}, \textsf{out\_mark}, \textsf{out\_data}. The subshift $X\subseteq A^H$ is the set of $x\colon H\to A$ such that $x(h)_{\mathsf{in\_mark}}=\bullet$ for at most one $h\in H$, that $x(h)_{\mathsf{out\_mark}}=\bullet$ for at most one $h\in H$, and that if $x(h_{\mathsf{in}})_{\mathsf{in\_mark}}=x(h_{\mathsf{out}})_{\mathsf{out\_mark}}$ then $x(hh_{\mathsf{in}})_{\mathsf{in\_data}}=x(hh_{\mathsf{out}})_{\mathsf{out\_data}}$; in words, there is at most one $\bullet$ in each mark component, and if there are $\bullet$s then the \textsf{in\_data} and \textsf{out\_data} are translates of each other in such a manner that their position relative to the $\bullet$ mark is the same. It is easy to see that $X$ is an effective subshift, merely using the solution of the word problem in $H$ to enumerate forbidden patterns.

  Let $Y$ be the induction of $X$ to $G$, and assume by contradiction that $Y$ is covered by a SFT $Z\subseteq B^G$ for some finite set $B$. Let $S$ be a memory set for $Z$, so that the forbidden patterns defining $Z$ are defined on (subsets of) the finite set $S$.

  Since $G$ is amenable, there exist by definition \emph{F\o lner sets}: finite subsets $F\subseteq G$ such that $\#(S F \setminus F)<(\log2/\log\#B)\,\#F$. Let $T$ be a right transversal of $H$ in $G$, and let $h\in H$ be such that $F\cap F h=\emptyset$ (such an $h$ exists because $H$ is infinite and $F^{-1}F$ is finite).

  Now for every ``pattern'' $p\colon F\to\{0,1\}$ we construct a configuration $y_p\colon G\to A$ as follows:
  \begin{align*}
    y_p(g)_{\mathsf{in\_mark}} &= \begin{cases} \bullet & \text{ if }g = 1_G,\\ \circ & \text{ otherwise},\end{cases}\\
    y_p(g)_{\mathsf{out\_mark}} &= \begin{cases} \bullet & \text{ if }g = h,\\ \circ & \text{ otherwise},\end{cases}\\
    y_p(g)_{\mathsf{in\_data}} &= \begin{cases} p(g) & \text{ if }g\in F,\\ 0 & \text{ otherwise},\end{cases}\\
    y_p(g)_{\mathsf{out\_data}} &= \begin{cases} p(g h^{-1}) & \text{ if }g\in F h,\\ 0 & \text{ otherwise}.\end{cases}
  \end{align*}
  Note that $y_p$ defines a configuration in $\iota_*(X)$, since on each coset the mark is translated by $h$ and the data (a slice of $F$, extended by $0$'s) is also translated by $h$. Note furthermore that the non-zero \textsf{in\_data} and \textsf{out\_data} never overlap, by our choice of $h$. Choose furthermore for every $y_p$ a lift $z_p\in Z$ under $Z\twoheadrightarrow Y$.

  There are $2^{\#F}$ patterns $p$, and $(\#B)^{\#(S F \setminus F)}<2^{\#F}$ restrictions of $z_p$ to $S F \setminus F$, so there are two patterns $p\ne q$ such that $z_p,z_q$ have the same restriction to $S F \setminus F$. We may ``paste them'': let $z$ be defined as $z_p$ on $F$ and $z_q$ on $G\setminus F$. On the one hand, $z\in Z$ since the local conditions of $Z$ are satisfied; but on the other hand the image of $z$ in $Y$ has \textsf{in\_data} $p$ and \textsf{out\_data} a translate of $q$; this violates the ``copy'' condition of $\iota_*(X)$.
\end{proof}
  
\section{Baumslag-Solitar groups}
We expect that a similar simulation theory could be developed on the amenable Baumslag-Solitar groups, whose Cayley graphs, instead of being horocyclic products of two trees, are horocyclic products of a tree and a hyperbolic tiling. In this section, we show that at least some simulation results can be obtained for Baumslag-Solitar groups by simulating the lamplighter group in them. We concentrate on the group
\[\BS \coloneqq \BS(1, 2) = \langle a, b \mid a^b = a^2 \rangle.\]

We recall an SFT from~\cite{esnay-moutot:bs}. For $c \in \{0, 1\}$ define the substitution $\tau_c\colon\{0, 1\} \to \{0, 1\}^2$ by $\tau_c(1) = 00$, $\tau_0(0) = 01$, $\tau_1(0) = 10$. Let us define an SFT $Y$ on $\BS$ over the alphabet $\{0, 1\} \times \{0, 1\}$ with the following allowed patterns on $\{1, a\}$:
\[ \{ (1 \mapsto (c, *); a \mapsto (c, *)) \mid c \in \{0, 1\} \}, \]
and the following allowed patterns on $\{1,a,b\}$:
\[ \{ (b \mapsto (c, d); 1 \mapsto (*, \tau_c(d)_0); a \mapsto (*, \tau_c(d)_1)) \mid c, d \in \{0, 1\} \} \]
(note that here we define two SFTs by allowed patterns and take their intersection).

A detailed study of this SFT can be found in~\cite{esnay-moutot:bs}*{\S3}. The crucial point is that on each $\langle a \rangle$-coset we have a substitutive configuration for one of the two substitutions, and the $\langle a\rangle$-cosets above (= translated by $b$) contain its desubstitutions under $\tau_0$ and $\tau_1$ respectively. Here is a typical picture of a configuration in $Y$:
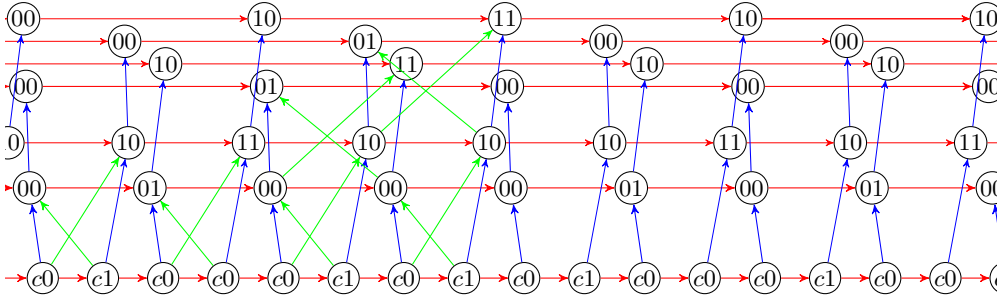
\begin{figure}[h!]
  \centerline{\begin{tikzpicture}[x={(0.8,0)},y={(0.1,1.5)},z={(0.3,0.2)},every node/.style={circle,draw,inner sep=0.5pt},>=stealth']
      \clip (-0.6,-0.2) rectangle (15.6,2.5);
      \foreach\i/\c in {0/0,1/1,2/0,3/0,4/0,5/1,6/0,7/1,8/0,9/1,10/0,11/0,12/0,13/1,14/0,15/0,16/0} {
        \node (a\i) at (\i,0,0) {\small $c\c$};
      }
      \foreach\i/\c/\d in {-1/??/10,0/00/10,1/01/11,2/00/10,3/00/10,4/00/10,5/01/11,6/00/10,7/01/11,8/00/??} {
        \node (b0\i) at (2*\i,1,-1) {\small $\c$};
        \node (b1\i) at (2*\i+1,1,1) {\small $\d$};
      }
      \foreach\i/\c/\d/\e/\f in {-1/??/??/??/00,0/00/00/10/10,1/01/01/11/11,2/00/00/10/10,3/00/00/10/10,4/00/??/??/??} {
        \node (c0\i) at (4*\i,2,-1.5) {\small $\c$};
        \node (c1\i) at (4*\i+1,2,0.5) {\small $\d$};
        \node (c2\i) at (4*\i+2,2,-0.5) {\small $\e$};
        \node (c3\i) at (4*\i+3,2,1.5) {\small $\f$};
      }
      \foreach\i/\j in {0/1,1/2,2/3,3/4,4/5,5/6,6/7,7/8,8/9,9/10,10/11,11/12,12/13,13/14,14/15,15/16} \draw[red,->] (a\i) -- (a\j);
      \foreach\i/\j in {b00/b01,b01/b02,b02/b03,b03/b04,b04/b05,b05/b06,b06/b07,b07/b08,c00/c01,c01/c02,c02/c03,c03/c04} \draw[red,->] (\i) -- (\j);
      \foreach\i/\j in {b1-1/b10,b10/b11,b11/b12,b12/b13,b13/b14,b14/b15,b15/b16,b16/b17,c10/c11,c11/c12,c12/c13,c20/c21,c21/c22,c22/c23,c32/c33,c3-1/c30,c30/c31,c31/c32,c32/c33} \draw[red,very thin,->] (\i) -- (\j);
      \draw[red,<-] (a0) -- +(-1,0,0);
      \draw[red,<-] (b00) -- +(-1,0,0);
      \draw[red,<-] (c00) -- +(-1,0,0);
      \draw[red,very thin] (b17) -- +(1,0,0) (c13) -- +(3,0,0) (c23) -- +(2,0,0);
      \draw[red,very thin,<-] (c10) -- +(-4,0,0);
      \draw[red,very thin,<-] (c20) -- +(-4,0,0);
      \draw[red,very thin,<-] (c3-1) -- +(-4,0,0);
      \foreach\i/\j in {a0/b00,a2/b01,a4/b02,a6/b03,a8/b04,a10/b05,a12/b06,a14/b07,a16/b08,b00/c00,b02/c01,b04/c02,b06/c03} \draw[blue,->] (\i) -- (\j);
      \foreach\i/\j in {a1/b10,a3/b11,a5/b12,a7/b13,a9/b14,a11/b15,a13/b16,a15/b17,b01/c20,b03/c21,b05/c22,b07/c23,b1-1/c3-1,b10/c10,b12/c11,b14/c12,b16/c13,b11/c30,b13/c31,b15/c32,b17/c33} \draw[blue,very thin,->] (\i) -- (\j);
      \foreach\i/\j in {a0/b10,a1/b00,a2/b11,a3/b01,a4/b12,a5/b02,a6/b13,a7/b03,b02/c21,b12/c31,b03/c01,b13/c11} \draw[green,->] (\i) -- (\j);
    \end{tikzpicture}}
  \caption{A configuration in $T$. The generators $a$ and $b$ of $\BS(1,2)$ are coloured respectively \textcolor{red}{red} and \textcolor{blue}{blue}, and the simulated generators $a,b$ of $\lamp$ are coloured respectively in \textcolor{green}{green} and \textcolor{blue}{blue}.}\label{fig:B}
\end{figure}

We can easily interpret every connected component from every configuration in this SFT as a Cayley graph of the lamplighter group: the generator $b$ of $\BS$ corresponds directly to the generator $b$ of $\lamp$, and the generator $a$ of $\lamp$ corresponds to either $a b$ or $a^{-1}b$, depending on the first bit the configuration at the \emph{end} of the edge. This is illustrated in Figure~\ref{fig:B}. In the terminology of~\cite{bartholdi-salo:ll}, we obtain the following:
\begin{prop}
  There exists a $\BS(1,2)$-SFT such that all its configurations simulate the lamplighter group.\qed
\end{prop}

Every configuration in $Y$ defines a graph, which generically is the Cayley graph of $\lamp$. In countably many cases, the graph will consist of two disjoint copies of the Cayley graph of $\lamp$.

We now detail the construction of a $\BS$-SFT starting from a $\lamp$-SFT. Let $\Theta\subseteq A^{\{e, a, b^{-1} a, b\}}$ be a set of forbidden tetrahedra on the lamplighter group, defining an SFT $Z$. Define an SFT $X \subset Y \times A^B$ by adding to $Y$ the following forbidden patterns:
\[ \{ (1 \mapsto (*, *, z); b \mapsto (0, *, x); a \mapsto (*, *, w); ba \mapsto (1, *, y)) \mid \forall (x,y,z,w) \in T \}. \]
It is easy to see that the natural projection $\pi_Y\colon X \to Y$ is surjective, and for a configuration $y \in Y$, there is a natural correspondence between $Z$ and the configurations of $X$ in $\pi^{-1}(y)$. We deduce the following results.
\begin{thm}
  The entropies of $\BS(1, 2)$-SFTs are the nonzero $\Pi^0_1$ real numbers.
\end{thm}
\begin{proof}
  Because with probability $1$ every Cayley graph of $\BS$ simulates a single Cayley graph of $\lamp$, and the subshift $Y$ has zero entropy, we obtain exactly the entropy of $X$ by the above construction.
\end{proof}

\begin{thm}
  Let $\psi\colon\BS(1, 2) \to \Z$ be the natural projection. Then the $\psi$-pullback of every effective $\Z$-system has an SFT cover.
\end{thm}
\begin{proof}
  Let $X$ be an effective $\Z$-system. In the construction above, sea-levels in $\lamp$ correspond exactly to sea-levels (cosets of $\Z[\frac12]$) in $\BS$, and furthermore the $\hor$-cosets in $\sea$ are contained in $\langle a\rangle$-cosets in the sea-level of $\BS$.

  In the generic case that the Cayley graph of $\BS$ simulates the Cayley graph of $\lamp$, the SFT cover of $\phi^* X$ to $\lamp$ directly gives, via the simulation $Y$, an SFT cover of $\psi^*X$. We have to add one extra SFT rule to make sure that, if there are two copies of $\lamp$'s Cayley graph in a $\BS$-Cayley graph, they both represent the same configuration of $X$. The construction in~\S\ref{sec:InfinitelyMany} encodes this information from $X$ in a Toeplitz subshift in the $\ver$ direction, and the extra SFT rule on $\BS$ enforces this Toeplitz subshift to be constant on $\langle a\rangle$-cosets.
\end{proof}

\section{Open questions}
For concreteness, we have concentrated on the lamplighter group, but the methods generalize directly to at least some groups. In particular, there is no difficulty in extending all the constructions to $K \wr \Z$ for a finite group $K$.

On the other hand, it would be useful to have more flexible tools to understand subshifts on groups such as $\Z \wr \Z$ or $K \wr \Z^2$, for which our method does not immediately extend. Definitely the undecidability of the domino problem and existence of aperiodic SFT's are not in doubt, but the questions of which effective systems pull back to sofic shifts may be more subtle.

There is a general construction, that of \emph{horocyclic products}, which may be the natural direction for generalizing our construction. Consider two graphs $\mathscr G_1,\mathscr G_2$ each endowed with a \emph{Busemann function}: a height function $\eta_i\colon V(\mathscr G_i)\to\Z$ that varies at most by $1$ along edges; and form the graph $\mathscr G$ with vertex set $\{(v_1,v_2)\in V(\mathscr G_1)\times V(\mathscr G_2)\mid\eta_1(v_1)+\eta_2(v_2)=0\}$, and with an edge from $(v_1,v_2)$ to $(w_1,w_2)$ whenever ($v_1=v_2$ or there is an edge from $v_1$ to $v_2$) and  ($w_1=w_2$ or there is an edge from $w_1$ to $w_2$).

Considering $\mathscr G_1=\mathscr G_2$ the ternary tree with $\eta(v)$ the ``distance to infinity'' along a chosen ray in the tree, we obtain the Cayley graph of $\lamp$. There are other interesting examples: take as before for $\mathscr G_1$ the ternary tree, and let $\mathscr G_2$ be the ($1$-skeleton of the) tiling of $\mathbb H^2$ by hyperbolic pentagons with vertex set $2^\Z(\Z+i)$; the Busemann function is $\eta(2^n(m+i))=n$. Then the horocyclic product is essentially the Cayley graph of the Baumslag-Solitar group $\BS$, which fits in an exact sequence
\[\begin{tikzcd} 1\arrow{r}& \Z[1/2]\arrow{r}& \BS(1,2)\arrow{r}{\psi}& \Z\arrow{r}& 1\end{tikzcd}\]
(More precisely, half the edges have to be shifted one position horizontally.) Here we showed that one can interpret this as the lamplighter group to translate results from the lamplighter group to it, but perhaps a better method would be to have a uniform construction for such horocyclic products.

A uniform approach would certainly help in understanding induction in the context of $\BS(1,2)$. The induction of an $\langle a \rangle$-subshift is not necessarily $\BS$-sofic by Observation~\ref{obs:effectivenonsofic}. However, the analogue of Theorem~\ref{thm:effectiveH} to $\BS$ would ask to pull back effective $\Z[1/2]/\Z$-subshifts to $\Z[1/2]$ and then induce them to $\BS$, expecting to obtain sofic $\BS$-subshifts in this manner.

It is also possible to define the horocyclic product of more than two graphs: given graphs $\mathscr G_i$ for $i\in\{1,\dots,n\}$, each with a Busemann function $\eta_i$, their horocyclic product has as vertex set $\{(v_1,\dots,v_n)\in V(\mathscr G_1)\times\cdots\times V(\mathscr G_n)\mid \eta_1(v_1)+\cdots+\eta_n(v_n)=0\}$, with an edge from $(v_1,\dots,v_n)$ to $(w_1,\dots,w_n)$ when $v_i=w_i$ for all $i$ except at most $2$, which are joined by an edge.

For example, the horocyclic product of three copies of the ternary tree is the Cayley graph of a finitely presented group containing the lamplighter group; see~\cite{baumslag:fpmetabelian} and~\cite{bartholdi-n-w:horo}.

We finish with some questions about possible dynamics of SFTs. The first one may be solvable by combining our techniques with existing fixed-point techniques, but certainly requires more thought:
\begin{question}
  Is there a minimal strongly aperiodic SFT on the lamplighter group?
\end{question}
Our construction is heavily based on a very rigid ``odometric skeleton'', and our systems always have a non-trivial equicontinuous factor. This prevents many types of dynamical behaviors. In the case of $\Z^d$ with $d \geq 2$ one can perform the fixed-point construction without a rigid skeleton by allowing macrotiles of less rigid shapes, see e.g.~\cite{westrick:tcpe}. In our lamplighter case, we have no idea how to perform a similar feat, and expect, for example, the following question to be difficult (here, fancy fixed-point tricks are not helpful, as the skeleton is already in place when we start using the fixed-point method).

For a group $G$, we call a $G$-system \emph{topologically mixing} if for any nonempty open sets $U, V$, there exists a finite subset $C \subset G$ such that $\forall g \notin C: gU \cap V \neq \emptyset$.
\begin{question}
  Is there a topologically mixing aperiodic SFT on the lamplighter group?
\end{question}

\begin{question}
  In Remark~\ref{rem:CannotBe} we show that the subshift $Y=\iota_*\eta^*(\{0,1\})=\phi^*(\{0,1\}^\Z)$ is not AFT, even though $Y\times\wee$ is AFT. Is there a finite-to-one cover of $Y$ that is AFT?
\end{question}
Our transducer construction relies on subshifts such as $Y$, for which we introduced $\wee$; the construction would be simplified without it.

\appendix
\section{Implementation details for remembering single bits}\label{sec:Implementation}

\subsection{Universal tileset}

In this section, we describe a tileset that allows the simulation of arbitrary Turing machines. One can think of this tileset as being a ``universal'' (non-deterministic four-headed) Turing machine which simulates arbitrary deterministic two-headed Turing machines; universality is in quotes, because typically a universal machine simulates all machines of the same type. Though the machine is non-deterministic, it is unambiguous, meaning that every row of the tiling has a unique ``successor'' row above it.

In the informal descriptions, we will talk about tiles receiving and sending signals, since many of the signals can be thought of as being sent due to some event, and then triggering actions elsewhere. Formally, these signals are realized as local consistency rules, since the ``computation'' takes place instantaneously.

The tileset Turing machine has four heads $A,B,C,D$; we call $A,B$ the \emph{simulator} heads, and $C,D$ the \emph{simulated} heads.

\begin{rem}
  The reason for using multiple heads is that two simulator heads $A,B$ allows for an easy and quick  simulation of steps of the simulated machine (we simulate one step of an $n$-state machine in $O(\log n)$ steps). Furthermore, the implementation of the final algorithm will be easier with two heads $C,D$, and in particular we find that the ``magic trick'' of reading the program from the tape is somewhat clearer with two simulated heads. The downside of having a relatively complicated tileset (instead of a very simple universal Turing machine) is that programming it in itself takes some work. Nevertheless, this simply means checking an explicit list of constraints, which is straightforward even if lengthy.
\end{rem}

Since the tileset has to eventually be implemented on a Turing machine, and all but the lowest level of the construction will actually always be using it in encoded form, we directly work in terms of codings.

\newcommand\BNFterm[1]{\ensuremath{\langle\texttt{#1}\rangle}}
\newcommand\XNScolour{\BNFterm{NS colour}}
\newcommand\Xbit{\BNFterm{bit}}
\newcommand\Xtrit{\BNFterm{trit}}
\newcommand\Xphase{\BNFterm{phase}}
\newcommand\Xsymbol{\BNFterm{symbol}}
\newcommand\XA{\BNFterm{A}}
\newcommand\XB{\BNFterm{B}}
\newcommand\XC{\BNFterm{C}}
\newcommand\XD{\BNFterm{D}}
\newcommand\Xz{\texttt0}
\newcommand\Xo{\texttt1}
\newcommand\Xt{\texttt2}
\newcommand\Xa{\texttt@}
\newcommand\Xh{\texttt\#}
\newcommand\Xp{\texttt\%}
\newcommand\Xd{\texttt\$}
We use the alphabet $S = (\{0,1,2\} \times \{0,1,2,3\}) \sqcup \{\Xa,\Xh,\Xp,\Xd\}$, and fix an encoding of $S$ in $4$ bits: $(2a_1+a_0,2b_1+b_0) \in \{0,1,2\} \times \{0,1,2,3\}$ is encoded as $a_1a_0b_1b_0$, and $(\Xa,\Xh,\Xp,\Xd)$ are respectively encoded as $1100, 1101, 1110, 1111$. We refer to elements of $\{\Xa,\Xh,\Xp,\Xd\}$ as \emph{special symbols}. We also identify trits $\Xz,\Xo,\Xt$ with the respective elements $(0,0),(1,0),(2,0)$ of $S$. We strive to use \texttt{teletype} to represent elements of $S$. Thus every element of $S$ automatically has an encoding in $S^4$, see the tag \Xsymbol\ below.

We will encode by $S^9$ the south and north colours of the universal tileset, and by $S^{63}$ the west and east colours (in practice, these symbols come from a much smaller subalphabet). The parsing of the $9$ symbols used as south and north colours is described by the following grammar (we roughly use Backus normal form enhanced with regular expressions and mathematical notation in this and all descriptions that follow).
\begin{bnftable}
  \XNScolour & \Xphase\ \Xsymbol\ \XA\ \XB\ \XC\ \XD\\
  \Xbit & \Xz\mid\Xo\\
  \Xtrit & \Xz\mid\Xo\mid\Xt\\
  \Xphase & \Xtrit\\
  \Xsymbol & \Xbit^4\\
  \XA,\XB,\XC,\XD & \Xbit
\end{bnftable}

The interpretations are as follows: The phase trit is consistent over rows, i.e.\ is shared by the south colours of every tile in the same row. Each simulated computation step will consist of first taking a single step in phase $0$, in which we read the instruction encoded on the tape, write new symbols, and move the heads; then taking one or more steps in phase $1$ in which the simulator heads jump to the next instruction and check that this guess is correct. Phase $2$ will be entered when a special symbol is read by the simulator head, and signifies acceptance. We will refer to the various components of the colour by the nonterminals used in the grammar, e.g.\ \Xphase, \Xsymbol, etc.

We will be considering tilings of square areas of size $2^n \times 2^n$ (eventually, this universal tileset will be used on one of the layers of a macrotile, and the $2^n \times 2^n$ area is precisely the macrotile), and we also describe some additional constraints that should hold on the borders. We implicitly require that everything conforms to the grammars; again this is enforced by finitely many local rules.

On the tiles comprising the north border of the square area, the only constraint is that the \Xphase\ trit of the south colour must contain the symbol \Xt.

\newcommand\Xtape{\BNFterm{tape}}
\newcommand\Xprog{\BNFterm{program}}
\newcommand\Xinstr{\BNFterm{instr}}
\newcommand\Xcmd{\BNFterm{cmd}}
\newcommand\Xid{\BNFterm{id}}
\newcommand\XCDinstr{\BNFterm{CDinstr}}
\newcommand\Xfail{\BNFterm{fail}}
\newcommand\Xok{\BNFterm{ok}}
\newcommand\XCread{\BNFterm{Cread}}
\newcommand\XDread{\BNFterm{Dread}}
\newcommand\XCwrite{\BNFterm{Cwrite}}
\newcommand\XDwrite{\BNFterm{Dwrite}}
\newcommand\XCmove{\BNFterm{Cmove}}
\newcommand\XDmove{\BNFterm{Dmove}}
\newcommand\XJL{\BNFterm{JL}}
\newcommand\XML{\BNFterm{ML}}
\newcommand\Xstay{\BNFterm{stay}}
\newcommand\XMR{\BNFterm{MR}}
\newcommand\XJR{\BNFterm{JR}}
\newcommand\Xdata{\BNFterm{data}}

We now describe the constraints on the south border of a valid tiling of a $2^n \times 2^n$ square, which are meant to begin the computation. The \Xphase\ bit is required to be \Xz\ in all the south colours on the south border. The south border \Xsymbol s, when concatenated, form a word in $((\Xz\mid\Xo)^4)^{2^n}$ and therefore in $S^{2^n}$ by the prescribed encoding of $S$; and this word should obey \Xtape\ in the following grammar, for some number $t$ depending on the simulated tileset (and in particular not fixed):
\begin{bnftable}
  \Xtape & \Xa\ \Xprog\ \Xp\ \Xdata\ \Xd\ \Xt^*\\
  \Xprog & \Xinstr\ \Xcmd^*\\
  \Xcmd & \Xh\ \Xid\ (\Xinstr\mid\Xfail\mid\Xok)\\
  \Xid & \Xbit^t\\
  \Xinstr & \XCDinstr\ \Xid\\
  \XCDinstr & \XCread\ \XCwrite\ \XCmove\ \XDread\ \XDwrite\ \XDmove\\
  \Xfail    & \Xt\Xz\\
  \Xok      & \Xt\Xo\\
  \XCread,\XDread & \Xsymbol\\
  \XCwrite,\XDwrite & \Xsymbol\\
  \XCmove,\XDmove & \XJL\mid\XML\mid \Xstay\mid \XMR\mid \XJR \\
  \XJL      & \Xo\Xz\Xz\Xz\Xz\\
  \XML      & \Xz\Xo\Xz\Xz\Xz\\
  \Xstay    & \Xz\Xz\Xo\Xz\Xz\\
  \XMR      & \Xz\Xz\Xz\Xo\Xz\\
  \XJR      & \Xz\Xz\Xz\Xz\Xo\\
  \Xdata    & (S\setminus\{\Xa,\Xh,\Xd\})^* \text{(to be further restricted later)}
\end{bnftable}

The head $A$ should start at the bottom left corner on \Xa, $B$ one step to the right (in the beginning of the first \Xinstr), and heads $C,D$ on top of the leftmost \Xp. This just means that the \XA,\XB,\XC,\XD\ bits have value \Xo\ in exactly the tiles on these positions of the tape. The tags \XJL, \XML, \XMR, \XJR\ stand for ``jump left'', ``move left'', ``move right'', ``jump right'', respectively.

We start by a high-level description of the behavior of the machine, and in this manner explain the interpretation of the \Xcmd\ expressions in \Xtape. We use the symbol \Xh\ to separate commands. Each such command contains an identifier \Xid, which is just a bit string and serves as a unique identifier for a state of the simulated machine. The \Xid s should all be of the same length $t$ (though it would suffice for correct functioning that they form a prefix code). This is followed by \Xinstr, \Xfail\ or \Xok. The meaning of \Xfail\ is to halt the computation (by a tiling error), while \Xok\ initiates phase $2$. The more interesting (and typical) instruction is \Xinstr. Here, in state $0$ the west and east colours should memorize \XCDinstr, check it against the \Xsymbol s that the heads $C,D$ are seeing, and rewrite the symbols and move the heads according to these instructions (we do all of this in a single row), continuing with the new \Xid\ identifier. The meaning of \XML, \Xstay\ and \XMR is to move the head by $-1,0,1$ respectively. The meaning of \XJL\ and \XJR\ is that the head directly jumps to the nearest special symbol in the direction stated.

The behavior of the heads $A,B$ is as follows. At the beginning of phase $0$, the head $B$ should be standing on the symbol of some \Xcmd\ immediately to the right of the \Xh\ \Xid\ part, except in the very beginning, when it is on the first \Xinstr; this is in fact just a hack to get the computation started. The sweep reading \XCDinstr\ should drop the $A$ head at the start of the \Xid\ to the right of \XCDinstr. The $A$ and $B$ heads that were in the south colours when entering phase $0$ simply disappear (or one may think that $A$ ``jumps'' over \XCDinstr; new $B$s will be born non-deterministically). Phase $0$ always takes just one step, and the next phase after phase $0$ is always $1$. In the beginning of phase $1$, any number of $B$ heads appear to the right of \Xh-symbols. In the beginning of phase 1, we should then have the $A$ head at the beginning of some \Xid\ at the end of an \Xinstr, and should have any number of $B$ heads on the tape. We will now check, using the west and east colours, that there is in fact exactly one $B$ head on the row (barring a tiling error), and that the \Xid\ it reads at the beginning of the \Xcmd\ is precisely the \Xid\ read by $A$. This takes exactly $t$ steps. Note that there is no marker that tells us when we are done in the \Xcmd; we do not need one because the $A$ head knows the \Xid\ is over when it hits a special symbol. (All the \Xid s are of the same length $t$, but our tileset cannot not know what this $t$ is, as we are describing a fixed tileset.)

After phase $1$ (meaning $A$ is about to move to a special symbol) we enter phase $0$ (if we guess the computation continues) or phase $2$. In the latter case, we check that there is an \Xok\ on the tape at the $B$ head, and we simply stay in phase $2$ until the end of the computation. There is no special behavior to account for encountering a \Xfail, so encountering it leads to a tiling error (as desired). (We could alternatively simply enter a loop, and thus hit the north border without being in state $2$.)

We now explain the west and east colours in more detail, and the rules that govern them. This will also make some of the statements from the previous paragraphs more precise. The west and east colours follow the following grammar:

\newcommand\XWEcolour{\BNFterm{WE colour}}
\newcommand\Xphasez{\BNFterm{phase 0}}
\newcommand\Xphaseo{\BNFterm{phase 1}}
\newcommand\Xphaset{\BNFterm{phase 2}}
\newcommand\Xginstr{\BNFterm{ginstr}}
\newcommand\Xtinstr{\BNFterm{tinstr}}
\newcommand\XAbirth{\BNFterm{A birth}}
\newcommand\XBbirth{\BNFterm{B birth}}
\newcommand\XCmsig{\BNFterm{C mvsig}}
\newcommand\XDmsig{\BNFterm{D mvsig}}
\begin{bnftable}
\XWEcolour & \Xz\ \Xphasez\mid \Xo\ \Xphaseo\mid \Xt\ \Xphaset\\
\Xphasez   & \Xtinstr\ \Xginstr\ \XCmsig\ \XDmsig\ \XAbirth\ \XBbirth\\
\Xginstr   & \XCDinstr\\
\Xtinstr    & ((\Xbit^*\ \Xt^+)\cap\Xtrit^{26})\mid\XCDinstr\\
\XAbirth   & \Xbit\\
\XBbirth   & \Xbit\\
\XCmsig    & \Xbit^4\\
\XDmsig    & \Xbit^4\\
\Xphaseo   & \Xtrit^2\ \XAbirth\ \XBbirth\ \Xt^{58}\\
\Xphaset   & \Xtrit\ \Xt^{61}
\end{bnftable}
Here, \Xginstr, \Xtinstr\ stand respectively for ``guessed instruction'' and ``true instruction'', \XCmsig\ stands for ``$C$ movement signal''. Note that we reuse some nonterminals from the previous grammars.

The first symbol of the colour tells us the phase, and the first constraint is that it matches the phase in the south symbol.

Let us start by explaining the constraints of phase $2$, which is the easiest. Here, we suppose that $B$ is standing on top of an $\Xok = \Xt\Xo$ (if $B$ is on a bit, we should be in phase $0$). We should thus check that the symbol under $B$ is \Xt, and we should check that the next symbol is \Xo. This is checked using the first trit of \Xphaset: on the west border we require that this first trit is \Xt. In a cell with $B$ in it (i.e.\ with $\XB = \Xo$ in the south colour) we check that it receives \Xt\ from the west, has \Xt\ as the tape symbol (i.e. we have $\Xsymbol=\texttt{1000}$ in the south colour), and sends \Xo\ to the east. We check that a cell receiving \Xo\ from the west has no $B$, has \Xo\ as tape symbol in the south colour, and sends \Xz\ to the east. Cells receiving \Xz\ from the west must not contain $B$-heads. We check that any cell receiving \Xt\ or \Xz\ from the west just sends it to the east if $B$ is not in the cell. At the east end we check that the trit is \Xz. As for the south and north colours, we copy all the rest of the content verbatim. We also check that head $A$ is on a special symbol (though this is not strictly necessary for correct functioning).
 
It clearly follows that we can tile a row in phase $2$ if and only if there is a unique $B$-head, and it is standing on the leftmost symbol of a \Xt\Xo-subword. Other phases will fail in such a situation (as we see below) so we must then preserve the phase, and in fact copy this row until the northmost row (where there are no forbidden patterns, since the north border accepts rows in phase $2$). Note that we check for phase $2$ on the south border of the northmost row, so that there is at least one row that checks the $B$-on-a-\Xt\Xo\ constraint and we are not just non-deterministically sending a phase $2$ signal to the north at the last minute.

We now proceed to phase $1$. Here the first two bits are used to compare what the $A$ and $B$ heads are seeing. The \XAbirth\ and \XBbirth\ signals are used to give birth to a new head. On the west end, we require \Xt\Xt\ in the first two trits, and \Xz\Xz\ in \XAbirth\XBbirth. At a cell containing the head $A$, we check that the first trit is \Xt\ and change it to the bit coded in the \Xsymbol\ in the south colour (and we check that it is indeed a bit). The \XAbirth\ bit in the east colour is \Xo\ if and only if we have $A$ in this cell. At a cell containing $B$ we do the same but for the second trit. Cells that receive the \XAbirth\ or \XBbirth\ signal from the west will send an $A$ or $B$ head to the north. Finally, at the east end we check that the first two trits are \Xz\Xz\ or \Xo\Xo.

It is clear that we we can tile a row in phase $1$ if and only if $A,B$ are on equal bits (and there is exactly one head of each type), and in the row above we will have $A$ and $B$ heads entering the cells one step to the right. As for the south and north colours, cells not receiving \XAbirth\ or \XBbirth\ from the west will send no $A,B$ heads north. Heads $C,D$ are just copied directly, and so are all tape symbols. The next row can be in any phase (though only one guess will succeed), so we make no requirements on the north phase bits.

In phase $0$, which is the most complicated one, we check that $A$ is on a special symbol, $B$ is on a bit, and we read the instruction to the right of $B$, and according to this instruction, read and write with $C,D$ and change the tape. More precisely, we have the following constraints. On the west end, \Xtinstr\ must contain $\Xt^{26}$, \XCmove\ and \XDmove\ must both contain \texttt{0000}, \XAbirth, \XBbirth\ must both contain \Xz\ (intuitively meaning, no heads come from, or enter, the non-existing cell to the left of the leftmost one).

While we have $\Xt^{26}$ in \Xtinstr, we simply copy it from west to east. When we encounter $B$, we start reading bits from the current south \Xsymbol\ (checking they are indeed bits), always writing them on the leftmost \Xt\ in \Xtinstr\ before sending the adjusted colour to the east. When we rewrite the last \Xt\ to a bit, we have successfully read a \XCDinstr\ from the tape and we send an \XAbirth\ signal to the east (i.e.\ set $\XAbirth = 1$ on the east colour). Once \Xtinstr\ contains no \Xt s, we again simply copy it from west to east on each step. Also, we check that once \Xtinstr\ is no longer $\Xt^{26}$ we do not see any $B$ heads. On the east end we check $\Xtinstr = \Xginstr$. If head $C$ (respectively $D$) is present, we check that the \Xsymbol\ on the south colour is equal to \XCread\ (respectively \XDread); if not, we produce a tiling error. We may (but don't have to) send \XBbirth\ signals to the east from any cell with south $\Xsymbol$ equal to \Xh, i.e.\ we only impose $\XBbirth=0$ when the south \Xsymbol\ is not \Xh.

In the north and south directions, we by default simply copy the current tape symbol from south to north, except if $C$ or $D$ is present; we then instead write the symbol in \XCwrite\ and/or \XDwrite\ (and if both $C$ and $D$ are present and require conflicting writes we trigger a tiling error). We introduce $A$ and/or $B$ if and only if \XAbirth\ and/or \XBbirth\ is present on the west. As for $C,D$, we do the same. We send a $C$ north if and only if there is a $C$ head in the current cell and \XCmove\ of \Xginstr\ contains \texttt{00100} (the stay command); or there is no $C$, but there is a \texttt{0010} in \XCmsig\ meaning ``move right'', coming in from the west; or there is no $C$, but there is a \texttt{0100} in \XCmsig\ meaning ``move left'', coming in from the east; or \texttt{0001} meaning ``jump right'' (respectively \texttt{1000} meaning ``jump left'') in \XCmsig\ on the west (respectively east) and the current cell contains a special symbol. If there is no special symbol, we just propagate the signals \texttt{1000} or \texttt{0001} forward, checking that west and east contents agree.

This concludes the high-level (but complete) description of the tileset. 

We now explain how to use the tileset to perform universal computation of any deterministic two-headed Turing machine (one can also perform non-deterministic computation with this tileset, but we do not, and the claims about unambiguity of course need not be correct if we do).

The machines we simulate have two heads ($C,D$) which cannot sense each other's presence directly, but have shared state, and can also accept or fail by moving to a special state. They can read the tape content and rewrite it (under the cell), move around, and jump to the next or previous special symbol. If the machine to be simulated has $k$ states, we pick $t = \lceil \log k \rceil$ and pick a binary string of length $t$ for each state. After the initial \Xa-symbol, let us agree to put an instruction of the form `$\texttt{1110} \; \texttt{1110} \; \texttt{00100} \; \texttt{1110} \; \texttt{1110} \; \texttt{00100} \; \Xz^t$', where $\Xz^t$ is the identification of the initial state; this means that, when it is booted, the machine checks that heads $C,D$ are above a \Xp\ and starts at state $\Xz^t$. Then, for each transition (in the machine to be simulated) from state $q$ to state $q'$ where head $C$ reads symbol $r$, head $D$ reads symbol $r'$, the heads write $w,w'$ respectively, and the move instructions (stays, moves or jumps) are $j, j'$, we include
\begin{equation}\label{eq:cmd}
  \Xh\ [q]\ [r]\ [w]\ [j]\ [r']\ [w']\ [j']\ [q']
\end{equation}
as one of the \Xcmd s in \Xtape, where the brackets denote the natural encoding in the grammar of \Xtape, namely $j\in\XCmove$, $j'\in\XDmove$, etc. For an accept state, we use an `\Xh\ \Xok' type \Xcmd. To fail, we can similarly use `\Xh\ \Xz', or we can simply not simulate the transition at all; both choices lead to a tiling error.

\subsection{The fixed-point construction in detail}
We list some conventions we follow and numerical choices:
\begin{itemize}
\item For simplicity we assume that the simulated subshift $X$ is over binary alphabet.
\item The responsibility zones of macrotiles of level $k$ are of length $2^k$ and the $i$th tile is responsible for the $\lfloor i/2^k \rfloor$th such zone, which of course corresponds to a natural $\hor_k$-block. In this manner, many tiles are responsible for the same block, but this leads to particularly simple calculations.
\item Instead of a computation zone at the center, we simply start the computation at the south border of the macrotile, and the entire macrotile is used as a computation area.
\item The neighbour colours are stored to the right of the program (starting after the first \Xp\ symbol), in order $N,S,W,E$ or $S,N,W,E$ depending on the parity of the y-co\"ordinate of the position of the macrotile in its parent. (The first thing we will do in our program is to swap these back in software if the y-co\"ordinate is odd, after which the order is always $N,S,W,E$.)
\item Tiles will be referred to as macrotiles of level $0$. For $k>0$, macrotiles of level $k$ will consist of $2^{n_k}$ tiles of level $k-1$, with $n_k = 4^{k + c}$ for a suitable constant $c$. We expect $c = 1$ to suffice, meaning macrotiles of level $1$ already have side length $65536$. The actual size of the macrotile of level $k$ (measured in macrotiles of level $0$) is $L_k = \prod_{1\leq i \leq k} 2^{4^{i+c}} \leq 2^{4^{k+1+c}}$, meaning there are more macrotiles of level $k$ in a macrotile of level $k+1$ than there are tiles in a macrotile of level $k$. Note that $L_k$ is a power of $2$, so a macrotile is of the shape $\hor_\ell \times \ver_\ell$ for some $\ell$.
\item We define $t_k = 4^{k + 4 + c}$; this will be used as the length of ``packets'' exchanged between macrotiles.
\item We assume that the uniquely decodable substitution $\tau\colon A \to A^2$ has $\#A=2$, so elements of $A \times A$, encoding $\ver$ and $\hor$ configurations of $\tau$, can be conveniently encoded as $\Xsymbol$s in $(\Xz\mid\Xo)^4$.
\item Great care is taken so that computations are done in tiny morsels; in particular, each cell only takes care of a small part of the checking.
\end{itemize}

The macrotile of level $k$ consists of macrotiles of level $k-1$. The macrotiles of level $k-1$ should be thought of as a tileset with two layers. From the point of view of level $k$, its level-$(k-1)$ macrotiles come from a tileset of constant size, but with complex rules governing the allowed patterns. At level $1$, this is simply enforced by SFT rules, but in general the job of the $(k-1)$-macrotiles is to ensure the correct structure on level $k$. The first layer of a level-$(k-1)$ macrotile simulates a tile of the universal tileset from the previous section, and the second carries individual symbols along ``wires''.

\newcommand\XNSWE{\BNFterm{NSWE}}
\newcommand\XN{\BNFterm{N}}
\newcommand\XS{\BNFterm{S}}
\newcommand\XW{\BNFterm{W}}
\newcommand\XE{\BNFterm{E}}
\newcommand\Xk{\BNFterm{k}}
\newcommand\Xc{\BNFterm{c}}
\newcommand\Xpos{\BNFterm{pos}}
\newcommand\Xppos{\BNFterm{ppos}}
\newcommand\Xword{\BNFterm{word}}
\newcommand\Xpword{\BNFterm{pword}}
\newcommand\Xsub{\BNFterm{subst}}
\newcommand\Xpsub{\BNFterm{psubst}}
\newcommand\Xxpos{\BNFterm{xpos}}
\newcommand\Xypos{\BNFterm{ypos}}
\newcommand\Xpxpos{\BNFterm{pxpos}}
\newcommand\Xpypos{\BNFterm{pypos}}
\newcommand\Xpacket{\BNFterm{packet}}
\newcommand\Xwire{\BNFterm{wire}}
\newcommand\Xsimu{\BNFterm{simu}}
We are ready to explain how we program the universal Turing machine of the previous section. The data portion in the south colours of the $(k-1)$-macrotiles on the southmost row of a $k$-macrotile is split into a fixed number ($4\cdot9+10=46$) of \Xp-separated tritfields. They are of the following form, with the number of \Xt's at the end of \Xtape\ adjusted so the total length of the expression is $2^{n_k}$, and the number of \Xt's at the end of \Xpacket\ adjusted so the total length of the packet is $t_k$:
\begin{bnftable}
  \Xdata & \begin{array}[t]{l}
                 \XNSWE\ \Xp\ \Xk\ \Xp\ \Xc\ \Xp\ \Xpos\ \Xp\ \Xppos\\
                 \Xp\ \Xword\ \Xp\ \Xpword\ \Xp\ \Xsub\ \Xp\ \Xpsub
               \end{array}\\
  \XNSWE & \XN\ \Xp\ \XS\ \Xp\ \XW\ \Xp\ \XE\\
  \XN,\XS,\XW,\XE & \Xpacket\cap S^{t_k}\\
  \Xk & \Xz^k\\
  \Xc & \Xz^c\\
  \Xpos & \Xxpos\ \Xp\ \Xypos\\
  \Xxpos,\Xypos & (\Xz\mid\Xo)^{4^{k+1+c}}\\
  \Xppos & \Xpxpos\ \Xp\ \Xpypos\\
  \Xpxpos,\Xpypos & (\Xz\mid\Xo)^{4^{k+2+c}}\\
  \Xword & (\Xz\mid\Xo)^{2^k}\mid \Xt^{2^k}\\
  \Xpword & (\Xz\mid\Xo)^{2^{k+1}}\mid \Xt^{2^{k+1}}\\
  \Xsub,\Xpsub & \Xsymbol\\
  \Xpacket & \Xpword\ \Xp\ \Xpos\ \Xp\ \Xppos\ \Xp\ \Xwire\ \Xp\ \Xsimu\ \Xp\ \Xpsub\ \Xp\ \Xt^*\\
  \Xwire & (\Xz\mid\Xo\mid\Xt)^4\\
  \Xsimu & \XNScolour\ \Xt^{54}\mid \XWEcolour
\end{bnftable}

Here \Xppos, \Xpword\ refer to the position and word of a macrotile's parent (one level higher). The ``beams'' or ``packets'' sent along the \Xpacket\ regions are of great importance, and will be detailed later. If a beam points to outside of a macrotile, \Xpword\ is simply the $0$-word. \Xsimu\ serves to simulate the universal tileset, by propagating \XNScolour\ or \XWEcolour\ which follows the grammar of the previous section.

We turn to the program \Xprog\ to be run on our universal Turing machine, ultimately a list of \Xcmd\ as in~\eqref{eq:cmd}, namely the transcription of a two-headed deterministic Turing machine into the format understood by the universal tileset. Instead of giving the transitions, we explain the (many, easy) feats that the Turing machine must perform. Our task is to write a program that checks that, if every level-$k$ macrotile executes correctly \Xprog\ according to its position \Xpos, then every level-$(k+1)$ macrotile also follows \Xprog, and its data is consistent with its children in the obvious way (the $k$-level macrotiles directly contain some information about the parent, such as the parent word \Xpword, and we simply check the consistency of these).

This is at the heart of the self-reference we wish to implement: the program and data are encoded in the \Xsymbol\ fields of the south border of each level-$k$ macrotile, namely words in $S^*$. Each macrotile also behaves as a single tile, and has in each (N,S,W,E) direction a ``beam'', which contains in particular a symbol (its \Xwire)


First, the Turing machine will swap the north and south neighbour data beams if the vertical position \Xypos\ is odd. This (like everything else) is easy enough that one can directly describe how the simulated two-headed Turing machine does it: we use $32$ right jumps to get $D$ to the \Xp-symbol just after \Xpos, then move one step backwards to read its parity, and branch to different states of the Turing machine. If the parity is even, we jump $C$ and $D$ leftwards (again through a fixed number of \Xp s) to get back to the initial symbol. Otherwise, we jump $C$ and $D$ leftwards (still a fixed number of \Xp s) to the beginnings of the \XS\ beam and \XN\ beam, respectively, and then have both heads step right. Now, in a single sweep over these areas, we can swap their contents using our two heads.

Next, we perform a sequence of consistency checks. First of all, we check that in the \Xsimu\ areas of all packets we have an element of the universal tileset. This means that our Turing machine must check the logic described in the previous section. If the \Xpos\ of our macrotile is on one of the borders, we also check the border conditions for the tiles of the universal tileset. This is just a finite list of conditions, which was implicitly listed in the previous section. 

We describe some of the details involved in the above paragraphs. First of all, we must be able to check whether we are on a boundary; for this we must check whether the \Xxpos\ (respectively \Xypos) of our own data is equal to $0$ or to $2^{n_{k+1}} - 1 = 2^{4^{k+1+c}} - 1$, as the present macrotile of level $k$ is a cell in a level $k+1$ macrotile which is of size $2^{4^{k+1+c}}$. This amounts to checking if the word \Xxpos\ (respectively \Xypos) equals $\Xz^{4^{k+1+c}}$ or $\Xo^{4^{k+1+c}}$.

On the boundaries we perform some constant number of checks. The only difficulty is the south border, where the checks refer to specific positions of the heads. We check that $\Xphase = \Xz$ (easy), that $A$ is present (in the packet \XS) if and only if $\Xxpos = 0$, that $B$ is present if and only if $\Xxpos = 1$, and that $C,D$ are present if and only if we are on the first \Xp. We will (below) ensure that the program is the same, so it is easy to figure out where the \Xp\ is, just having $C$ walk along the program and increment a counter with $D$, and check if this counter value is \Xxpos\ when $C$ reaches \Xp. Note that we can have $D$ write the counter directly on top of the \Xxpos\ area, since our alphabet is $\{0,1,2\} \times \{0,1,2,3\}$, so it is easy to compare these numbers. (It is in such places that it is convenient to have a second component in $S$, giving us effectively a scratch area along every field of \Xdata).


We are now ready to discuss the fixed-point part of the construction. What the macrotiles of level $k$ need to ensure is that the the data on the south row in \Xsimu\ in the macrotile of level $k+1$ is as expected, and wire bits are propagated correctly. Then the next macrotile will also be sending information correctly and performing the calculations on the correct data, and we get by induction that all levels work similarly. We start with \Xsimu. Let us describe things checked on the south border, namely what the level $k$ macrotile checks about the symbols in its south beam \XS\ in case $\Xypos = 0$. We describe the checks in order from left to right.

We can check that the program is correct in the same way as we checked for the \Xp-symbol after it when placing $C,D$: while walking along the program with $C$, if the counter incremented by $D$ hits the value in \Xxpos, we check that the south packet has the same symbol on the tape as $C$ is reading. The symbols that are part of the program are nothing special, so $C$ can read them just fine, and it is easy to locate \Xsimu\ of the south packet for comparison. (This is sometimes called the ``magic trick'' of the fixed-point argument.) Also, we check that if we are on the position exactly after the program (where we enforced $C,D$ already), then the symbol is indeed \Xp.

If \Xxpos\ belongs to the \XNSWE\ area of the next macrotile, we just check that the symbols belong to $\{\Xz,\Xo,\Xt,\Xp\}$, and \Xp's appear in the correct positions. (There are many consistency checks on the values, but these are checked by the macrotile of level $k+1$, as described later, so the macrotiles of level $k$ need not worry about them.) To go over this area and perform the checks, we must increment our counter by $4t_{k+1} + 4$. It is easy to increment a scratch binary counter (stored using the second component of $S$ within \Xxpos) by $t_{k+1} = 4^{k + 4 + c} = 2^{2k+2c+8}$: simply use one head to walk over the \Xk\ and \Xc\ areas twice, then perform $8$ more steps, and then add $1$. Each time the binary counter is incremented, check that if it equals \Xxpos\ then the symbol in the south colour is \Xp, and otherwise it is a trit.

We then check that the parent \Xk\ and \Xc\ values are correct, namely $\text{parent}\Xk = k+1$ and $\text{parent}\Xc = c$. For this, beginning with the counter value we already have on \Xxpos\ (which should be where the \Xk-area begins in the parent tile) we have the $C$ head walk on the tape over \Xk\ while incrementing the counter. At the end we take one extra step to increase by one the \Xk-value in the parent tile. Then we do the same with \Xc. Again if the counter value equals \Xxpos\ then we check we have the correct value \Xz\ or \Xp\ in the \Xsymbol\ of \Xsimu\ of \XS.

Next we check that that the parent's own position on tape is correct. If \Xxpos\ represents a bit string in this area, we simply check this bit string against \Xdata.\Xppos. (The operator $.$ refers to indexing parts of a field, so this means the $\Xppos$-component inside the $\Xdata$ field.) What this means in practice is that we increment the counter while walking with $C$ on \Xdata.\Xppos, and if the counter equals \Xxpos\ then we check that \XS.\Xsimu.\Xsymbol\ has the same bit as our \Xdata.\Xppos. Next we skip over the parent.\Xppos\ area (by incrementing the counter by $4^{k+c+2}$), as we have no idea what its value should be, and only check that if \Xxpos\ is in this area, the symbol is a bit, except at the exact middle where it must be an \Xp. 

Next we similarly check the parent's \Xword, namely compare it with \Xpword. Then we check parent.\Xsub\ against \Xpsub. For parent.\Xpword\ and parent.\Xpsub\ again we only check the alphabet constraints.

We then check that at the end we have a \Xd\ symbol and then only \Xt s. This is straightforward, we just scan for (the encoding of) \Xd\ in \XS.\Xsimu.\Xsymbol\ if \Xxpos's value is exactly the expected length of \Xdata\ of the parent, and if it is larger, we check that the symbol encoded is \Xt.

We then check that the data in the beams \XNSWE\ is consistent with the data in \Xdata. First we check that
\[ \XE.\Xpos = \XN.\Xpos = \Xdata.\Xpos \equiv \XW.\Xpos+1 = \XS.\Xpos+1\pmod{2^{n_{k+1}}}.\]
We then check $\BNFterm{I}.\Xppos = \Xdata.\Xppos$ if $I \in \{\texttt N,\texttt S,\texttt W,\texttt E\}$ is an edge that is not on the boundary of the macrotile. We do the same for \Xpsub, and finally check that after all these entries there are only trailing \Xt s.

The only non-trivial part to check about the \XNSWE\ beams is that \Xwire\ contents are correct. What makes this non-trivial is that the wires make some turns, and we have to figure out whether the current macrotile is part of a wire or not. We follow the picture below (note that each macrotile checks that its \emph{parent} looks like this, by routing according to its position in this picture) when \Xpypos\ is even; when \Xpypos\ is odd, the beam that goes up is from the \XS\ zone.
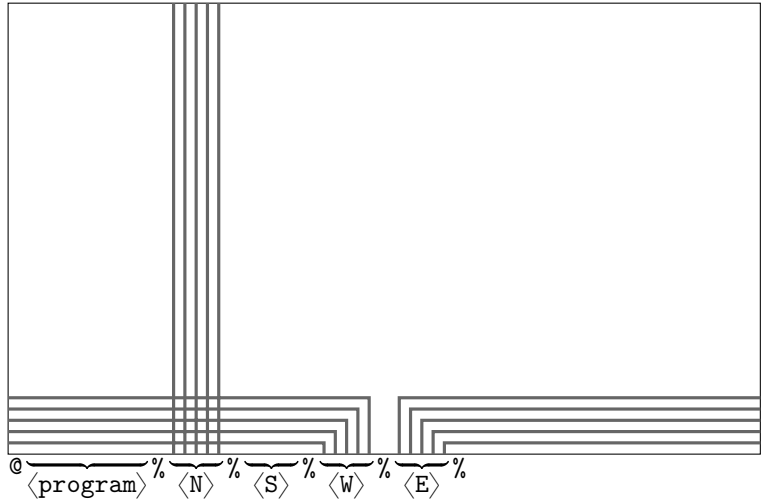
\begin{center}
  \begin{tikzpicture}
    \draw (0,0) rectangle (10,6);
    \node[anchor=north] at (0.1,0.08) {\Xa};
    \draw [decorate,very thick,decoration = {calligraphic brace}] (1.85,-0.1) -- node[below] {\Xprog} ++(-1.6,0);
    \node[anchor=north] at (2.0,0.08) {\Xp};
    \foreach \s/\x in {\XN/2.85,\XS/3.85,\XW/4.85,\XE/5.85} {
      \draw [decorate,very thick,decoration = {calligraphic brace}] (\x,-0.1) -- node[below] {\s} ++(-0.7,0);
      \node[anchor=north] at (\x+0.15,0.08) {\Xp};
    }
    \foreach \x in {0,0.15,0.3,0.45,0.6} {
      \draw [very thick,black!60] (2.2+\x,0) -- ++(0,6)
          (4.2+\x,0) -- ++(0,\x+0.15) -- (0,\x+0.15)
          (5.2+\x,0) -- ++(0,0.75-\x) -- (10,0.75-\x);
    }
  \end{tikzpicture}
\end{center}
We now explain the arithmetic of making these signals move correctly. There are of course many alternative ways of routing the signals, but the one we chose has the advantage of being easy to implement in linear time. The following rules ensure that the \Xwire\ symbol is either the correct one transmitted along a wire, or \texttt{2222} to indicate ``no wire''.

First, let us attach the ends of wires correctly. We assume $\Xypos = 0$. Suppose first that \Xpypos\ is even. If \Xxpos\ points to the parent.\XN\ area (which we already know how to check) then our north packet's \Xwire\ bit should be equal to the \Xsimu.\Xsymbol\ of the south packet, and the south wire symbol should be \texttt{2222}. If \Xxpos\ points to the parent.\XS\ area then instead the south wire symbol should equal the south colour's \Xsimu.\Xsymbol\ and the north wire symbol should be \texttt{2222}. On the west and east, we copy data to the north wire bits. If \Xxpos\ points somewhere else and $\Xypos = 0$, then the north and south wire bits carry \texttt{2222}. If \Xpypos\ is odd we do the same but with the roles of north and south exchanged.

We then propagate the data along the wires, and make sure that there is no additional information passed on edges not part of wires (otherwise our extension would have additional entropy). If $\Xypos>0$ and \Xxpos\ is not in the parent.\XW\ or parent.\XE\ areas, then we just check that the north and south bits agree. For the west beam, if \Xxpos\ is in the parent.\XW\ area, we compute $h = \Xxpos - \text{parent}.\XW.\text{start}$, namely how far \Xxpos\ is within the parent.\XW\ area. If $\Xpypos<h$ then we copy the south wire symbol to north, and send \texttt{2222} west and east. If $\Xpypos = h$ then we copy south to west and send \texttt{2222} north and east. If $h < \Xpypos < t_{k+2}$ then we copy east to west and send \texttt{2222} south and north. If $\Xpypos \geq t_{k+1}$ then we send \texttt{2222} in all four directions. For the east beam, the same calculation is performed but pointing in the east direction.

We next check that the substitution is correctly implemented, namely that \Xdata.\Xsub\ is consistent with \Xdata.\Xpsub. For this we have to actually compute the $n_{k+1}$th power of the substitution on \Xpsub, and check that our \Xsub\ is indeed what is expected in the \Xpos\ position of $\tau^{n_{k+1}}(\Xpsub)$. In practice, we memorize \Xpsub\ in a state of our Turing machine, and put heads $C,D$ respectively at the beginnings of \Xxpos\ and \Xypos\ by using a constant number of jumps. Each time both heads read a symbol we know which of the four quadrants we stepped into, and we can deduce the substituted symbol. In this manner, after a single sweep over \Xxpos\ and \Xypos\ we know what the \Xsub\ symbol should be.

Finally, we make sure that all macrotiles have access to the same word in the subshift $X$ under consideration; this means that \Xword\ should be consistent with \Xpword. For this, let us introduce the notation $\lfloor m\rfloor_k$ for the integer $\lfloor m/2^k\rfloor2^k=m-(m\%2^k)$; in binary, this is obtained from the representation of $m$ by clearing its lowest $k$ bits. The present tile's \Xword\ contains by induction the symbols in the columns in the interval $[\lfloor\Xypos\rfloor_k,\lfloor\Xypos+1\rfloor_k[$ counting from the left border (along level-$0$ tiles of course), if this interval fits inside the tile. Otherwise we are not responsible for anything and simply require that \Xword\ is a string of \Xt s. Similarly the parent is responsible for the columns $[\lfloor\Xpypos\rfloor_{k+1},\lfloor\Xpypos+1\rfloor_{k+1}[$ inside itself. We first translate our interval to the parent level by adding $L_k\Xxpos$, recalling that $L_k = \prod_{1\leq i \leq k} 2^{4^{i+c}} \leq 2^{4^{k+1+c}}$ is the size of the present macrotile. This number is very easy to construct: we simply run a $4$-ary counter on the $\Xk\ \Xp\ \Xc$ area with one head while moving the other head left, and whenever we detect a new carry, we add a bit in that position (and propagate a possible carry).

In practice, since the responsibility zones on our level are delimited by multiples by $2^k$ and those on the parent level by multiples of $2^{k+1}$, we just need to check whether $L_k\Xxpos+\lfloor\Xypos\rfloor_k=\lfloor\Xpypos\rfloor_{k+1}$ or $L_k\Xxpos + \lfloor\Xypos+1\rfloor_k = \lfloor\Xpypos+1\rfloor_{k+1}$,
both of which are straightforward. To detect the former condition, we check whether \Xword\ equals the left half of \Xpword, and to detect the latter, we check whether \Xword\ equals the right half of \Xpword.

All of this was bookkeeping to make sure that macrotiles of all levels perform the same computation. We finally have some time to do actual work: we run $k$ steps of computation of the Turing machine defining the $\hor$-subshift $X$, to obtain at most $k$ forbidden patterns, all of length at most $k$. We make sure that none of the subwords of \Xword\ is one of these forbidden patterns.



\tableofcontents

\end{document}